\documentclass[a4paper,12pt]{article}
\usepackage[utf8]{inputenc}
\usepackage{url}
\usepackage{stil}
\usepackage{booktabs, arydshln}
\usepackage{multirow}
\usepackage{pifont}
\usepackage{authblk}
\usepackage[symbol]{footmisc}

\makeatletter
\def\adl@drawiv#1#2#3{%
        \hskip.5\tabcolsep
        \xleaders#3{#2.5\@tempdimb #1{1}#2.5\@tempdimb}%
                #2\z@ plus1fil minus1fil\relax
        \hskip.5\tabcolsep}
\newcommand{\hdashlinelr}{%
  \noalign{\vskip\aboverulesep
           \global\let\@dashdrawstore\adl@draw
           \global\let\adl@draw\adl@drawiv}
  \hdashline[1pt/1pt]
  \noalign{\global\let\adl@draw\@dashdrawstore
           \vskip\belowrulesep}}
\makeatother

\title{Sizes of Pentagonal Clusters in Fullerenes}
\author[1,2]{Nino Ba\v{s}i\'{c} \thanks{\textbf{Acknowledgement:} The work of T. P. and N. B. has been
  supported in part by ARRS Slovenia, Projects P1-0294 and
  N1-0032. In addition the work of T. P. was supported by Projects
  J1-7051 and J1-6720.}}
\author[3]{Gunnar Brinkmann}
\author[4]{Patrick~W.~Fowler}
\author[1,2]{Toma\v{z}~Pisanski~${}^*$}
\author[3]{Nico Van Cleemput}

\affil[1]{FAMNIT, University of Primorska, Koper, Slovenia}
\affil[2]{Institute of Mathematics, Physics and Mechanics, Ljubljana, Slovenia}
\affil[3]{Applied Mathematics, Computer Science and Statistics, Ghent University, Krijgslaan 281 S9, 9000 Gent, Belgium}
\affil[4]{Department of Chemistry, University of Sheffield, Sheffield S3 7HF, UK}
\date{\today}

\begin{document}

\nocite{*} 

\maketitle

\footnotetext[2]{\textbf{E-mails:} \url{nino.basic@famnit.upr.si} (N. Bašić),
\url{gunnar.brinkmann@ugent.be} (G. Brinkmann),\\
\url{p.w.fowler@sheffield.ac.uk} (P. W. Fowler),
\url{tomaz.pisanski@upr.si} (T. Pisanski),\\
\url{nico.vancleemput@gmail.com} (N. Van Cleemput)
}

\begin{abstract}
Stability and chemistry, both exohedral and endohedral, of fullerenes
are critically dependent on the distribution of their obligatory 12
pentagonal faces.  It is well known that there are infinitely many
IPR-fullerenes and that the pentagons in these fullerenes can be at an
arbitrarily large distance from each other.  IPR-fullerenes can 
be described as fullerenes in which each connected cluster of pentagons
has size $1$. In this paper we study the combinations of cluster
sizes that can occur in fullerenes and whether the clusters can be at an arbitrarily
large distance from each other. For each possible partition of the 
number $12$,  we are able to decide whether the partition describes the
sizes of pentagon clusters in a possible fullerene, and state whether the different clusters
can be at an arbitrarily large distance from each other.
We will prove that all partitions with largest cluster of size $5$ or less can occur in
an infinite number of fullerenes with the clusters at an arbitrarily large distance of each other,
that $9$ partitions occur in only a finite number of fullerene isomers
and that $15$ partitions do not occur at all in fullerenes.

\end{abstract}


\noindent
{\bf Math.\ Subj.\ Class.\ (2010):}
05C10, 
52B10, 
92E10  

\section{Introduction}
All classical fullerene isomers are constructed according to the same
basic recipe of twelve faces pentagonal and all others hexagonal, but
their properties vary significantly, depending on the distribution of
the pentagons within the otherwise hexagonal framework.  Relative
total energies of isomers of lower fullerenes at a fixed number of
carbon atoms follow a general trend of decrease with minimisation of
pentagon adjacencies, leading ultimately to the isolated-pentagon rule
(IPR)~\cite{atlas, krotoipr, texas}, according to which we would
expect the most stable isomer to have isolated pentagons.  Exceptions
to these trends are rare: the most stable isomer of C$_{50}$ is
predicted to be $D_3$ 50:270, which has one more than the minimum
mathematically achievable number of pentagon adjacencies for this
vertex count~\cite{ceulemans, zhao2005}; the most stable isomer of
C$_{62}$, a vertex count for which IPR isomers are unavailable, is a
non-classical cage with one heptagonal face, but even this cage has
the minimum number of pentagon adjacencies~\cite{ayuela}.

The
rules for fullerene anions and for fullerenes as parent cages in
endohedral metallic fullerenes (EMFs) are more complicated, involving
both charge and pentagon adjacencies~\cite{Rich,NatureChem}.
Typically, the parent cages of EMFs are not those found for native
bare fullerenes, and may include face sizes other than $5$ and $6$.
The exohedral and endohedral chemistries of these molecules, real or hypothetical, depend on
their detailed structure, but nevertheless the trends are determined
to a great extent by combinatorially defined properties, many of which
depend on the pentagon distribution~\cite{atlas}.  

Early discussion on the relative
stability of isomers of [60]fullerene, C$_{60}$, noted that the sole
experimental isomer was the first IPR fullerene and that the second
experimentally observed species was the next possible
isolated-pentagon fullerene, the unique isomer of C$_{70}$ with 12
disjoint pentagons~\cite{bucky, taylor}. Other studies considered models of
stability based on the numbers of pentagon pairs and fully fused
pentagon triples \cite{trina}.
 
Work on the range of validity of the
face-spiral conjecture~\cite{mano91} found unspirallable
fullerenes with various patches of pentagons in fused triples and
quadruples, or in close associations of pentagons derived from
them~\cite{yoshida3, yoshida1, yoshida2}. With the discovery of
nanotubes~\cite{iijima}, it became routine to consider very large
fullerenes in which two hemispherical portions were separated by an arbitrarily long portion of cylindrical graphene.  

Proposed
mechanisms of rearrangements between fullerenes (the Stone-Wales
transformation \cite{stonewales}) and growth or decay processes
mediated by inclusion/extrusion of C$_2$ (the Endo-Kroto
mechanism\cite{endokroto}) all depend on the presence of specific
local mutual dispositions of pentagons and hexagons in fullerenes and nanotubes~\cite{atlas}.  

For all these reasons, it is important to study the
clustering of pentagons in general fullerenes and to understand the
basic mathematical limitations on their combinations and mutual
separations.  This is the subject of the present paper.

\section{Basic definitions}

\emph{Fullerenes} are closed carbon-cage molecules that contain only pentagonal and
hexagonal rings. Each \ce{C} atom is bonded to exactly three other \ce{C} atoms.
Fullerenes can therefore be modelled as trivalent \emph{convex polyhedra} with
only pentagonal and hexagonal faces. Atoms correspond to vertices of the
polyhedron, bonds to edges of the polyhedron and rings to faces.
An equivalent approach is via graph theory:

\begin{definition}
\label{def:fullerenegraph}
A \emph{fullerene graph} is a plane cubic graph that contains only pentagonal and hexagonal
faces (including the outer face).
\end{definition}

Alternatively, a fullerene graph is the skeleton of a cubic convex
polyhedron with pentagonal and hexagonal faces.  Let the number of
vertices, edges and faces of a fullerene graph be denoted by $n$, $e$
and $f$, respectively. It is easy to show that there are exactly 12
pentagonal faces in each fullerene graph.  There exists at least one fullerene graph
for each even number $n \geq 20$ with the exception of $n = 22$~\cite{Grunbaum}.

A special class of fullerenes is that of the IPR (Isolated Pentagon
Rule) fullerenes. In an IPR fullerene no two pentagons share an
edge. Using various methods from quantum mechanics, Albertazzi et
al.~\cite{Albertazzi1999} provided evidence that pentagon adjacencies
in a fullerene give rise to significant energetic penalties, leading
to a trend to minimsation of pentagon adjacencies, and ultimately to
the isolated-pentagon rule. The implication is that the most stable
isomer of a bare, neutral fullerene will obey the IPR if this is
mathematically possible, and at any rate will be one with the smallest
mathematically possible number of pentagon adjacencies; as noted
earlier, exceptions to this rule of thumb are rare.  Here we
generalise the idea behind IPR fullerenes and introduce the following
definition.

\begin{definition}

Let $F$ be a fullerene.

\begin{itemize}
\item For two pentagons $P,P'$ in $F$ we write $P\sim P'$ if they share an
edge. The equivalence classes of the equivalence relation 
generated by these relations are called the \emph{pentagon clusters} 
or for short \emph{clusters} of $F$.

\item The \emph{distance}
$d(p_1,p_2)$ between two pentagons $p_1$ and $p_2$ in $F$ is the minimum
number $d$ such that there are faces $p_1=f_1, \dots ,f_{d+1}=p_2$, so that
for $1\le i \le d$ the face $f_i$
shares an edge with $f_{i+1}$.

\item The \emph{distance}
$d(C_1,C_2)$ of two pentagon clusters $C_1$ and $C_2$ in $F$ 
is defined as $d(C_1,C_2)=\min \{d(p_1,p_2) \mid p_1\in C_1, p_2\in C_2\}$.

\item If $F$ has at least two pentagon clusters, we define the 
\emph{separation number}
$s(F)$ as \linebreak
$\min\{d(C,C')\}$, where $C, C'$ are different pentagon clusters of $F$.

\item A \emph{PIP} (\emph{Pentagonal Incidence Partition}) 
$(s_1,s_2,\dots ,s_k)$ of the fullerene $F$,
denoted $\mathit{PIP}(F)$, is the sequence of sizes
of the pentagon clusters of $F$ in non-increasing order.
\end{itemize}
\end{definition}

Each \emph{PIP} defines a partition\index{partition} of the
number 12, that is:  a non-increasing sequence $(p_1,p_2,\dots ,p_k)$
which sums up to 12. Whenever we talk about a partition, it will always be a partition of the number 12.

IPR-fullerenes are fullerenes $F$ with $\mathit{PIP}(F)=(1,1,1,1,1,1,1,1,1,1,1,1)$.

The main goal of this paper is the classification of partitions of the number 12 with respect
to their occurrence as \emph{PIP}'s of  fullerenes.
There will be four classes:

\begin{enumerate}
\item[(a)] Partitions that are not pentagonal incidence partitions of fullerenes.
\item[(b)]  Partitions that are pentagonal incidence partitions of a positive finite number of fullerenes.
\item[(c)]  Partitions that are pentagonal incidence partitions of infinitely many fullerenes $F$, but
$s(F)$ is bounded by a constant.
\item[(d)]  Partitions that are pentagonal incidence partitions of fullerenes $F$ with
arbitrarily large $s(F)$.
\end{enumerate}

 This classification is summarised in Table~\ref{tab:summary}. For partitions in class
 (b) the number of fullerenes realizing the partition is given. Fullerenes that have a
pentagonal incidence partition of type (b) can be downloaded from the graph
database HoG (see \cite{HoG}). 
For example, fullerenes with partition $(9,3)$
can be found by searching for 
the keyword \verb+pentagon_cluster_9_3+, and 
analogously for the other partitions of type (b).

\begin{table}
\newcommand{\tick}{\ding{51}}
\scriptsize
\begin{minipage}[t]{0.5\linewidth}
\ 
\begin{center}
\begin{tabular}{b{2cm}cccc}
\toprule
 & \multirow{2}{*}{\rotatebox[origin=c]{-90}{\parbox{2cm}{\centering Impossible (a)}}} & \multirow{2}{*}{\rotatebox[origin=c]{-90}{\parbox{2cm}{\centering Finite number (b)}}} &\multicolumn{2}{c}{Infinite}\\
Partition & & & \rotatebox[origin=c]{-90}{\parbox{1.8cm}{\centering bounded $s(F)$ (c)}} & \rotatebox[origin=c]{-90}{\parbox{1.8cm}{\centering unbounded\\ $s(F)$ (d)}}\\
\midrule
$(12 )$& & 41\\
\hdashlinelr
$(11,1)$ & & 2\\
\hdashlinelr
$(10,2)$ & & 1\\
$(10,1,1)$ & & 1\\
\hdashlinelr
$(9,3)$ & & 2\\
$(9,2,1)$ & \tick \\
$(9,1,1,1)$ & \tick \\
\hdashlinelr
$(8,4)$ & & 16\\
$(8,3,1)$ & \tick \\
$(8,2,2)$ & \tick \\
$(8,2,1,1)$ & \tick \\
$(8,1,1,1,1)$ & \tick \\
\hdashlinelr
$(7,5)$ & & 69\\
$(7,4,1)$ & & 12\\
$(7,3,2)$ & & 1\\
$(7,3,1,1)$ & \tick \\
$(7,2,2,1)$ & \tick \\
$(7,2,1,1,1)$ & \tick \\
$(7,1,1,1,1,1)$ & \tick \\
\bottomrule
\end{tabular}
\end{center}
\end{minipage}
\begin{minipage}[t]{0.5\linewidth}
\ 
\begin{center}
\begin{tabular}{b{2cm}cccc}
\toprule
 & \multirow{2}{*}{\rotatebox[origin=c]{-90}{\parbox{2cm}{\centering Impossible (a)}}} & \multirow{2}{*}{\rotatebox[origin=c]{-90}{\parbox{2cm}{\centering Finite number (b)}}} &\multicolumn{2}{c}{Infinite}\\
Partition & & & \rotatebox[origin=c]{-90}{\parbox{1.8cm}{\centering bounded $s(F)$ (c)}} & \rotatebox[origin=c]{-90}{\parbox{1.8cm}{\centering unbounded\\ $s(F)$ (d)}}\\
\midrule
$(6,6)$ & & & & \tick \\
$(6,5,1)$ & & & \tick \\
$(6,4,2)$ & & & \tick \\
$(6,4,1,1)$ & & & \tick \\
$(6,3,3)$ & & & \tick \\
$(6,3,2,1)$ & & & \tick \\
$(6,3,1,1,1)$ & \tick \\
$(6,2,2,2)$ & \tick \\
$(6,2,2,1,1)$ & \tick \\
$(6,2,1,1,1,1)$ & \tick \\
$(6,1,1,1,1,1,1)$ & \tick \\
\bottomrule
\end{tabular}
\end{center}
\end{minipage}
\caption{A summary of classes of the partitions of the number 12. All partitions $(p_1,p_2,\dots ,p_k)$ with $p_1<6$ are
of type (d).}\label{tab:summary}
\end{table}

\begin{table}
{
\footnotesize
\begin{tabular}{llll}
\toprule
PIP & Group & Count & List of fullerenes \\
\midrule
\((12)\)  & \(C_1\) & 6  & 36:7, 38:7, 38:11, 38:14, 40:34, 42:37\\
 & \(C_2\) & 18  & 32:1, 32:4, 34:1, 34:4, \underline{34:5}, 36:10, 36:11, 36:12, \underline{38:17}, 40:11, 40:23, 40:35,\\
 & & & 40:36, 42:38, 42:43, 44:66, 44:81, 46:113\\
 & \(C_s\) & 1  & 34:3\\
 & \(D_2\) & 3  & 28:1, 36:5, 44:85\\
 & \(C_{2v}\) & 3  & 30:2, \underline{30:3}, 38:12\\
 & \(D_3\) & 1  & \underline{32:6}\\
 & \(C_{3v}\) & 1  & 34:6\\
 & \(D_{2d}\) & 1  & \underline{36:14}\\
 & \(D_{3h}\) & 2  & \underline{26:1}, 32:5\\
 & \(D_{3d}\) & 1  & 44:86\\
 & \(D_{6d}\) & 2  & \underline{24:1}, 48:186\\
 & \(T_d\) & 1  & \underline{28:2}\\
 & \(I_h\) & 1  & \underline{20:1}\\
 & \textbf{all} & \textbf{41} & \\
\hdashlinelr
\((11,1)\)  & \(C_s\) & 2  & 40:28, 42:42\\
\hdashlinelr
\((10,2)\)  & \(C_{2v}\) & 1  & 40:37\\
\hdashlinelr
\((10,1,1)\)  & \(D_{5d}\) & 1  & \underline{40:39}\\
\hdashlinelr
\((9,3)\)  & \(C_s\) & 1  & 44:71\\
 & \(C_{3v}\) & 1  & 38:16\\
 & \textbf{all} & \textbf{2} & \\
\hdashlinelr
\((8,4)\)  & \(C_1\) & 6  & 38:8, 42:15, 42:36, 46:58, 48:60, 48:86\\
 & \(C_2\) & 7  & 40:15, 40:18, 44:76, 48:46, 48:63, 48:170, 52:83\\
 & \(C_s\) & 2  & 46:28, 46:57\\
 & \(C_{2v}\) & 1  & 36:9\\
 & \textbf{all} & \textbf{16} & \\
\hdashlinelr
\((7,5)\)  & \(C_1\) & 52  & 36:3, 38:3, 38:4, 38:5, 40:4, 40:6, 40:12, 40:26, 42:2, 42:4, 42:10, 42:25, 42:29,\\
& & & 42:30, 42:44, 44:9, 44:10, 44:18, 44:41, 44:42, 44:48, 46:6, 46:15, 46:17, 46:45,\\
& & & 46:71, 46:105, 48:10, 48:20, 48:181, 48:182, 50:10, 50:12, 50:139, 50:140, 50:141,\\
& & & 50:142, 50:232, 50:235, 52:9, 52:117, 52:118, 52:183, 52:196, 54:32, 54:33, 54:134,\\
& & & 56:58, 56:295, 58:17, 58:18, 60:30\\
 & \(C_s\) & 17  & 34:2, 36:4, 36:8, 40:7, 40:13, 40:24, 42:12, 44:11, 44:84, 46:8, 48:75, 50:33, 54:19,\\
& & & 54:474, 58:240, 60:90, 64:53\\
 & \textbf{all} & \textbf{69} & \\
\hdashlinelr
\((7,4,1)\)  & \(C_1\) & 8  & 44:51, 46:27, 46:29, 46:30, 46:59, 48:106, 50:50, 52:166\\
 & \(C_s\) & 4  & 44:28, 44:54, 46:41, 54:101\\
 & \textbf{all} & \textbf{12} & \\
\hdashlinelr
\((7,3,2)\)  & \(C_s\) & 1  & 48:141\\
\bottomrule
\end{tabular}
}
\caption{Lists of fullerenes for partitions for which only finitely many fullerenes exist.
The isomers are listed grouped with respect to their symmetry group and given as x:y
with x the number of atoms and y the number in the spiral order (see \cite{atlas}).
Isomers with a minimum number of pentagon adjacencies for their number of atoms
are underlined. 
Except for the programs for generation and computing the symmetry group, the results
were confirmed by two independent programs.}
\label{tab:spiralnumbers}
\end{table}

\section{Partitions $\boldsymbol{(p_1,p_2,\dots ,p_k)}$ with $\boldsymbol{p_1>6}$}

\begin{definition}
\label{def:patch}\ 
\begin{itemize}
\item A \emph{5-6-patch} or, for short, a \emph{patch} 
is a 2-connected plane graph that contains (apart from the outer
face) only pentagonal and hexagonal
faces and where all vertices in the outer face have degree 2 or 3 and all other vertices have degree 3.

\item The \emph{boundary length} $b(P)$ of a patch $P$ is the length of the cycle that is the boundary of the outer face. 

\item The \emph{complement} $C^c$ of a pentagon cluster $C$ 
in a fullerene $F$ is the plane graph consisting of
all vertices and edges that belong to faces of $F$ not in $C$. 
This implies that edges contained in only one pentagon of $C$
and vertices contained in only one or two faces of $C$ are in $C$ and in $C^c$.
\end{itemize}
\end{definition}

The complement of a cluster can be disconnected, and is therefore not necessarily itself
a patch, but each component of the complement is
a patch. A simple consequence of the Euler formula is that in each patch with $p<6$ pentagons
there are more vertices with degree 2 included in the boundary of the outer face than vertices with degree 3. This implies that 
there is an edge in the boundary of the outer face where both endpoints have degree 2.
If we have two patches with $p<6$ pentagons, we can identify the two patches along 
edges where both endpoints have degree 2. The result will be a patch that has the same
number of faces as the two patches together, but with a shorter boundary than
the sum of the boundary lengths. Iterating this process we get the following remark:

\begin{lemma}\label{lem:discon}

Let $\{P_1,P_2,\dots ,P_k\}$ be a set of at least two patches with a total of $p<6$ pentagons.
Then there is a patch $P$ with as many pentagons but more hexagons than $P_1,P_2,\dots ,P_k$
and with $b(P)=\sum_{i=1}^k b(P_i)$.

\end{lemma}

\begin{proof}
We will prove that if for $1\le j <k$ there is a patch $P^j$ with as many pentagons and 
at least as many hexagons as $P_1,P_2,\dots ,P_j$ and $b(P^j)=\sum_{i=1}^j b(P_i)$, then there is 
a patch $P^{j+1}$ with as many pentagons and 
more hexagons than $P_1,P_2,\dots ,P_{j+1}$ and 
with $b(P^{j+1})=\sum_{i=1}^{j+1} b(P_i)$. 

For $j=1$ we can choose $P^1=P_1$, so assume $P^j$ is given for $1\le j <k$. By identifying
$P^j$ and $P_{j+1}$ along edges where both endpoints have degree 2, we get a patch
$\bar P^{j+1}$ with $b(\bar P^{j+1})=(\sum_{i=1}^{j+1} b(P_i))-2$ and as many pentagons and 
at least as many hexagons as $P_1,P_2,\dots ,P_{j+1}$. In the cyclic sequence of degrees
in the boundary of a patch there are as many maximal sequences of 2's as there
are maximal sequences of 3's and in a patch with less than six pentagons 
there are more 2's than 3's in total. The maximum
length of a sequence of 2's is four, so the average length of a maximal sequence of 3's 
in such a patch is less than four, and so there is a sequence of length 1, 2 or 3. 
If one adds a hexagon at a place with a maximal sequence of 3's of length 
$i\in \{1,2,3\}$ then the boundary length grows by $4-2i$. So the boundary stays the same,
shrinks by 2 or grows by 2. As the number of hexagons in a patch with $p<6$ pentagons
and boundary length $b(\bar P^{j+1})$ is bounded (see \cite{BorBriGre2003}), successively 
adding hexagons at a shortest maximal sequence of 3's must finally 
grow from $(\sum_{i=1}^{j+1} b(P_i))-2$ to $\sum_{i=1}^{j+1} b(P_i)$.
As in this process at least one hexagon was added, we have constructed a patch with boundary length
$\sum_{i=1}^{j+1} b(P_i)$ and more hexagons than $P_1,P_2,\dots ,P_{j+1}$.
\end{proof}

For a partition $(p_1,p_2,\dots ,p_k)$ with $p_1>6$, we will use Theorem~\ref{thm:hexnumbers}
from \cite{BorBriGre2003} 
and Lemma~\ref{lem:discon} to determine an upper bound
on the number of hexagons in the complement of the cluster with size $p_1$, and therefore in the fullerene.
In \cite{BorBriGre2003} it is proven that among all patches with $p\le 6$ pentagons and $h\ge 0$
hexagons, a patch constructed in a spiral fashion starting with the pentagons has the shortest
boundary. This gives the following lower bounds on the boundary length: 

\begin{theorem}[\cite{BorBriGre2003}]\label{thm:hexnumbers}
 Let $P$ be a patch with $p \leq 5$ pentagons and $h$ hexagons. Then
\begin{equation*}
b(P) \geq \begin{cases}
2 \left\lceil \sqrt{12h - 3}\, \right\rceil & \text{if }p = 0,\\
2 \left\lceil \sqrt{10h + \frac{25}{4}} + \frac{1}{2} \right\rceil - 1 & \text{if }p = 1,\\
2 \left\lceil \sqrt{8h + 6} \right\rceil & \text{if }p = 2,\\
2 \left\lceil \sqrt{6h + \frac{81}{4}} + \frac{1}{2} \right\rceil - 1 & \text{if }p = 3,\\
2 \left\lceil \sqrt{4h + 25} \right\rceil & \text{if }p = 4,\\
2 \left\lceil \sqrt{2h + \frac{113}{4}} + \frac{1}{2} \right\rceil - 1 & \text{if }p = 5.
\end{cases}
\end{equation*}
\end{theorem}

\begin{corollary}\label{thm:ub_hexagons}
Let \(F\) be a fullerene containing a pentagon cluster with $k \geq 7$ pentagons. Let \(h(F)\) be the number of hexagons in \(F\). Then
\begin{equation*}
h(F) \leq \begin{cases}
52 & \text{if }k = 7,\\
36 & \text{if }k = 8,\\
31 & \text{if }k = 9,\\
30 & \text{if }k = 10, 11, \text{or } 12.
\end{cases}
\end{equation*}
\end{corollary}

\begin{proof}
A pentagon cluster with $k$ pentagons has at least $k-1$ edges that have a pentagon of
the cluster on each side. This implies that the sum of the boundary lengths of all
patches in the complement is at most $5k-2(k-1)=3k+2$.

Solving the equations in Theorem~\ref{thm:hexnumbers} for the number of hexagons and 
applying Lemma~\ref{lem:discon} to show that these numbers are also bounds for
more than one patch in the complement, we get the given bounds for the numbers
of hexagons in a fullerene with a pentagon cluster with $7\le k \le 12$ pentagons.
\end{proof}

The largest number of hexagons can appear for the case with a pentagon
cluster of size 7, so we find that fullerenes containing a pentagon
cluster of size 7 or more have at most 124 vertices.  In order to
identify all partitions $(p_1,p_2,\dots ,p_k)$ with $p_1>6$
that are pentagonal incidence partitions of a fullerene,
we have only to check all fullerenes with up to 124
vertices.  We did this by computer.  We searched the output of the
program {\em fullgen} (see \cite{BrDr95_2}) by two independent programs
for fullerenes with the clusters of size at least 7.  The results of
these searches are summarised in Table~\ref{tab:summary}.

\section{Partitions $\boldsymbol{(p_1,p_2,\dots ,p_k)}$ with $\boldsymbol{p_1=6}$}

There are 18 possible clusters containing 6 pentagons. The 17
clusters that are patches were  
generated by the program {\em ngons} described in \cite{BriGoe2015} that is part of the package
{\em CaGe} described in \cite{CaGe2010}. 
The unique cluster with two boundary components was added by hand. For simplicity we will add
the central hexagon to that cluster and talk about 18 closed clusters that are all patches.
If $C$ is one of these 18 clusters, then $C$ occurs in a fullerene $F$ if and only if its closure
occurs in $F$.
The clusters are shown in Figure~\ref{fig:p6},
together with their boundaries embedded in the hexagonal lattice.

If one of these closed clusters occurs in a fullerene, the complement is a patch. The fact that the
boundary of these clusters forms a path between two parallel lines in the hexagonal lattice
that neither intersects itself nor the parallel lines (except at the endpoints) implies
that the cluster and the complement containing one or more clusters are contained
in one-sided infinite nanotubes (with, of course, the same nanotube parameters).
The nanotube parameters are also displayed in Figure~\ref{fig:p6}. For the definition
of tube parameters see \cite{BNP99} or \cite{Dresselhaus96}.

These 18 clusters correspond to twelve different tube parameters:
$$
(5,0), (3,3), (4,2), (5,1), (6,0),  (4,3), (5,2),  (6,1), (7,0),  (4,4), (5,3),\text{ and }(6,2).
$$
We will refer to this set of parameters as
\(T_6\).

\begin{theorem}
Let $6\geq p_2 \geq p_3 \geq \dots \geq p_k > 0$, where \( k\geq2 \), be natural numbers with 
\(\sum_{i=2}^{k}p_i=6\).
The following three statements are equivalent:
\begin{enumerate}
\item[(a)] There exists a tube cap with parameters in \(T_6\) and with pentagon clusters of sizes $p_2 \geq p_3 \geq \dots \geq p_k$.
\item[(b)]  There exists a fullerene with pentagonal incidence partition $(6,p_2,\dots,p_k)$. 
\item[(c)]  There exists an infinite number of fullerenes with pentagonal incidence partition \linebreak 
$(6,p_2,\dots,p_k)$. 
\end{enumerate}
Moreover, if a fullerene with pentagonal incidence partition $(6,p_2,\dots,p_k)$
exists, the partition is of type (d) if the partition is $(6,6)$ and otherwise of type (c).
\end{theorem}

\begin{proof}
If fullerenes with the given partition exist, the boundaries embedded in the hexagonal
lattice (see Figure~\ref{fig:p6}) show that the closed clusters as well as their complements
can be extended to one-sided infinite nanotubes. The caps of these nanotubes will then contain
the closed cluster or the other clusters. 
If, on the other hand, caps with parameters in \(T_6\) 
contain pentagon clusters of sizes $p_2 \geq p_3 \geq \dots \geq p_k$,
they can be combined with a 6-cluster with the same parameters and an arbitrary amount of
rings of hexagons separating them. If there are only two clusters, this implies that $s(F)$ is unbounded. If there are three
or more clusters, at least two of them are contained in a cap for which  (see \cite{BNP99})
the number of faces, and therefore also the distance between the clusters, is bounded.
\end{proof}

\begin{figure}[!htbp]
\centering
\subfigure[$(5, 0)$]{
\begin{tikzpicture}[scale=0.015,thick]
    \definecolor{marked}{rgb}{0.25,0.5,0.25}
    \coordinate (15) at (2.474515,19.054680);
    \coordinate (14) at (61.288230,0.000000);
    \coordinate (13) at (2.437442,80.889714);
    \coordinate (12) at (97.562557,50.027804);
    \coordinate (11) at (60.120481,60.982391);
    \coordinate (10) at (26.274328,49.990732);
    \coordinate (9) at (39.212233,32.196479);
    \coordinate (8) at (33.262280,13.790548);
    \coordinate (7) at (6.941613,49.972197);
    \coordinate (6) at (39.193698,67.784986);
    \coordinate (5) at (60.139017,38.980538);
    \coordinate (4) at (75.801667,27.618165);
    \coordinate (3) at (33.225209,86.190917);
    \coordinate (2) at (61.232622,99.999999);
    \coordinate (1) at (75.764596,72.363300);
    \draw [black] (15) to (7);
    \draw [black] (15) to (8);
    \draw [black] (14) to (4);
    \draw [black] (14) to (8);
    \draw [black] (13) to (3);
    \draw [black] (13) to (7);
    \draw [black] (12) to (1);
    \draw [black] (12) to (4);
    \draw [black] (11) to (1);
    \draw [black] (11) to (5);
    \draw [black] (11) to (6);
    \draw [black] (10) to (6);
    \draw [black] (10) to (9);
    \draw [black] (10) to (7);
    \draw [black] (9) to (5);
    \draw [black] (9) to (8);
    \draw [black] (6) to (3);
    \draw [black] (5) to (4);
    \draw [black] (3) to (2);
    \draw [black] (2) to (1);
\end{tikzpicture}
\includegraphics[scale=0.4]{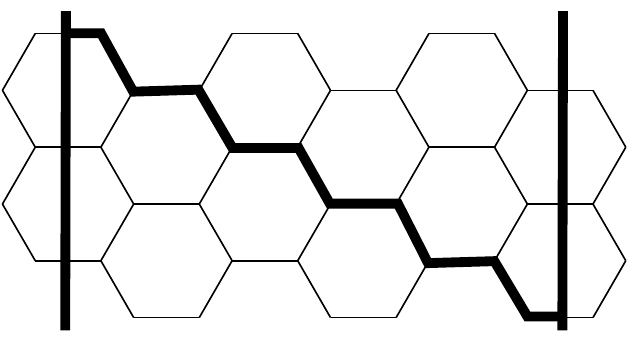}
}
\quad
\subfigure[$(3, 3)$]{
\begin{tikzpicture}[scale=0.015,thick]
    \definecolor{marked}{rgb}{0.25,0.5,0.25}
    \coordinate (16) at (16.108182,6.701821);
    \coordinate (15) at (0.000000,34.648749);
    \coordinate (14) at (83.908827,6.650792);
    \coordinate (13) at (99.999999,34.631739);
    \coordinate (12) at (33.832284,93.349208);
    \coordinate (11) at (50.008505,64.755910);
    \coordinate (10) at (37.557408,12.944377);
    \coordinate (9) at (32.862732,34.937914);
    \coordinate (8) at (62.646708,12.944377);
    \coordinate (7) at (67.239326,34.988942);
    \coordinate (6) at (16.108182,50.008504);
    \coordinate (5) at (50.008505,44.871576);
    \coordinate (4) at (83.959857,50.008504);
    \coordinate (3) at (28.627317,71.729884);
    \coordinate (2) at (66.116686,93.332199);
    \coordinate (1) at (71.389693,71.712876);
    \draw [black] (16) to (15);
    \draw [black] (16) to (10);
    \draw [black] (15) to (6);
    \draw [black] (14) to (13);
    \draw [black] (14) to (8);
    \draw [black] (13) to (4);
    \draw [black] (12) to (2);
    \draw [black] (12) to (3);
    \draw [black] (11) to (1);
    \draw [black] (11) to (5);
    \draw [black] (11) to (3);
    \draw [black] (10) to (9);
    \draw [black] (10) to (8);
    \draw [black] (9) to (5);
    \draw [black] (9) to (6);
    \draw [black] (8) to (7);
    \draw [black] (7) to (4);
    \draw [black] (7) to (5);
    \draw [black] (6) to (3);
    \draw [black] (4) to (1);
    \draw [black] (2) to (1);
\end{tikzpicture}
\includegraphics[scale=0.4]{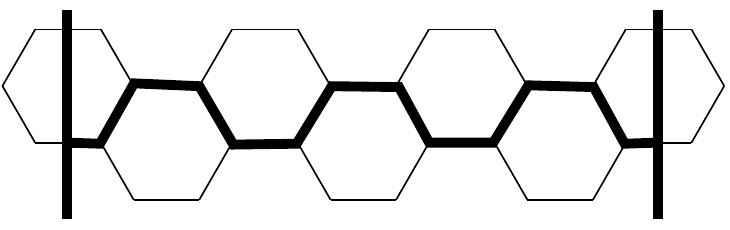}
}
\quad
\subfigure[$(4, 2)$]{
\begin{tikzpicture}[scale=0.015,thick]
    \definecolor{marked}{rgb}{0.25,0.5,0.25}
    \coordinate (16) at (41.825886,0.000000);
    \coordinate (15) at (70.916796,5.885978);
    \coordinate (14) at (85.138675,56.302003);
    \coordinate (13) at (14.861325,43.713405);
    \coordinate (12) at (29.021572,94.114021);
    \coordinate (11) at (40.593220,69.368259);
    \coordinate (10) at (59.391372,30.616332);
    \coordinate (9) at (40.392912,17.812017);
    \coordinate (8) at (77.388291,30.169492);
    \coordinate (7) at (58.174115,49.568567);
    \coordinate (6) at (29.052388,32.788905);
    \coordinate (5) at (41.825886,50.416024);
    \coordinate (4) at (70.978429,67.211094);
    \coordinate (3) at (22.611711,69.845917);
    \coordinate (2) at (58.127889,99.999999);
    \coordinate (1) at (59.591680,82.203389);
    \draw [black] (16) to (15);
    \draw [black] (16) to (9);
    \draw [black] (15) to (8);
    \draw [black] (14) to (4);
    \draw [black] (14) to (8);
    \draw [black] (13) to (3);
    \draw [black] (13) to (6);
    \draw [black] (12) to (2);
    \draw [black] (12) to (3);
    \draw [black] (11) to (1);
    \draw [black] (11) to (5);
    \draw [black] (11) to (3);
    \draw [black] (10) to (7);
    \draw [black] (10) to (8);
    \draw [black] (10) to (9);
    \draw [black] (9) to (6);
    \draw [black] (7) to (4);
    \draw [black] (7) to (5);
    \draw [black] (6) to (5);
    \draw [black] (4) to (1);
    \draw [black] (2) to (1);
\end{tikzpicture}
\includegraphics[scale=0.4]{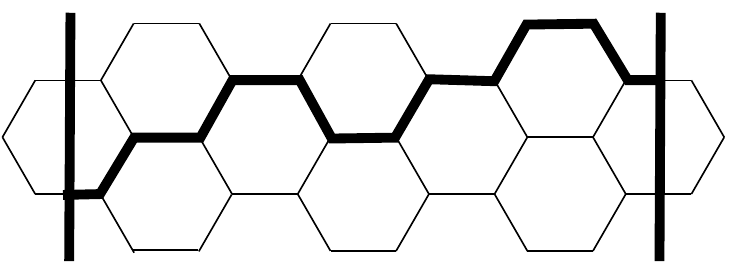}
}
\quad
\subfigure[$(5, 1)$]{
\begin{tikzpicture}[scale=0.015,thick]
    \definecolor{marked}{rgb}{0.25,0.5,0.25}
    \coordinate (17) at (38.271268,0.000000);
    \coordinate (16) at (20.902418,24.150584);
    \coordinate (15) at (74.653439,17.504755);
    \coordinate (14) at (79.097582,56.170155);
    \coordinate (13) at (22.383798,77.194889);
    \coordinate (12) at (38.733350,99.999999);
    \coordinate (11) at (32.223430,67.980429);
    \coordinate (10) at (46.602337,25.754280);
    \coordinate (9) at (55.232399,9.431911);
    \coordinate (8) at (58.018483,37.496602);
    \coordinate (7) at (29.287849,35.539548);
    \coordinate (6) at (74.639848,36.626801);
    \coordinate (5) at (51.345474,52.025007);
    \coordinate (4) at (32.495243,51.997825);
    \coordinate (3) at (64.351726,64.813809);
    \coordinate (2) at (61.185104,90.717588);
    \coordinate (1) at (53.954879,76.678446);
    \draw [black] (17) to (16);
    \draw [black] (17) to (9);
    \draw [black] (16) to (7);
    \draw [black] (15) to (6);
    \draw [black] (15) to (9);
    \draw [black] (14) to (3);
    \draw [black] (14) to (6);
    \draw [black] (13) to (12);
    \draw [black] (13) to (11);
    \draw [black] (12) to (2);
    \draw [black] (11) to (1);
    \draw [black] (11) to (4);
    \draw [black] (10) to (7);
    \draw [black] (10) to (8);
    \draw [black] (10) to (9);
    \draw [black] (8) to (5);
    \draw [black] (8) to (6);
    \draw [black] (7) to (4);
    \draw [black] (5) to (3);
    \draw [black] (5) to (4);
    \draw [black] (3) to (1);
    \draw [black] (2) to (1);
\end{tikzpicture}
\includegraphics[scale=0.4]{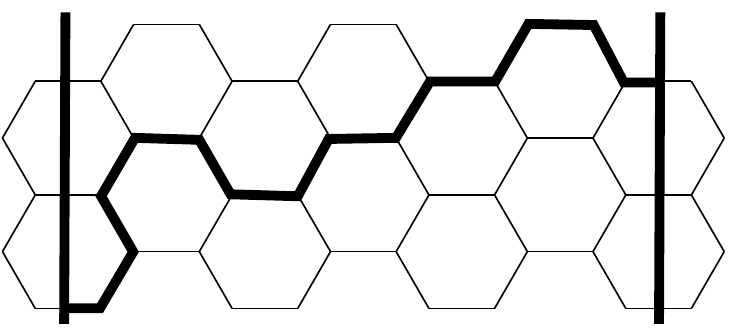}
}
\quad
\subfigure[$(6, 0)$]{
\begin{tikzpicture}[scale=0.015,thick]
    \coordinate (18) at (99.999999,50.047717);
    \coordinate (17) at (82.885318,31.072052);
    \coordinate (16) at (75.091460,6.720217);
    \coordinate (15) at (50.055671,12.064577);
    \coordinate (14) at (25.035788,6.688403);
    \coordinate (13) at (17.162399,31.008430);
    \coordinate (12) at (0.000000,49.984094);
    \coordinate (11) at (66.263719,40.631462);
    \coordinate (10) at (50.023860,31.199299);
    \coordinate (9) at (33.752189,40.583744);
    \coordinate (8) at (33.736282,59.368538);
    \coordinate (7) at (17.130589,68.927947);
    \coordinate (6) at (24.988072,93.263878);
    \coordinate (5) at (74.948306,93.311596);
    \coordinate (4) at (82.853509,69.007476);
    \coordinate (3) at (66.279626,59.416255);
    \coordinate (2) at (49.992048,68.784794);
    \coordinate (1) at (49.944330,87.919517);
    \draw [black] (18) to (17);
    \draw [black] (18) to (4);
    \draw [black] (17) to (16);
    \draw [black] (17) to (11);
    \draw [black] (16) to (15);
    \draw [black] (15) to (14);
    \draw [black] (15) to (10);
    \draw [black] (14) to (13);
    \draw [black] (13) to (12);
    \draw [black] (13) to (9);
    \draw [black] (12) to (7);
    \draw [black] (11) to (10);
    \draw [black] (11) to (3);
    \draw [black] (10) to (9);
    \draw [black] (9) to (8);
    \draw [black] (8) to (7);
    \draw [black] (8) to (2);
    \draw [black] (7) to (6);
    \draw [black] (6) to (1);
    \draw [black] (5) to (4);
    \draw [black] (5) to (1);
    \draw [black] (4) to (3);
    \draw [black] (3) to (2);
    \draw [black] (2) to (1);
\end{tikzpicture}
\includegraphics[scale=0.4]{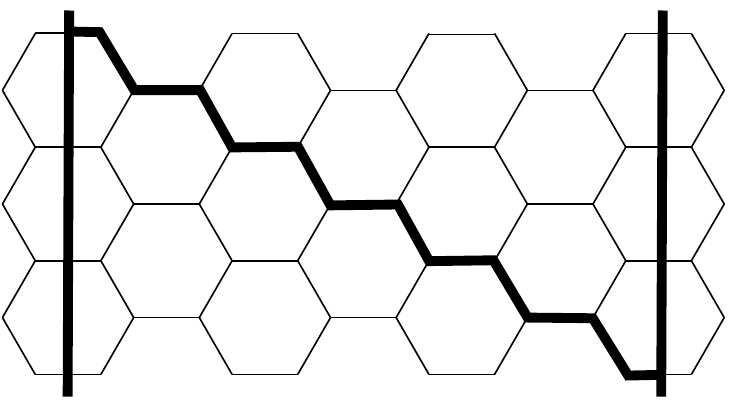}
}
\quad
\subfigure[$(6, 0)$]{
\begin{tikzpicture}[scale=0.015,thick]
    \definecolor{marked}{rgb}{0.25,0.5,0.25}
    \coordinate (19) at (22.491762,9.014675);
    \coordinate (18) at (0.014973,22.671457);
    \coordinate (17) at (8.415694,47.663969);
    \coordinate (16) at (73.554955,0.000000);
    \coordinate (15) at (99.386044,5.166214);
    \coordinate (14) at (99.985026,31.491462);
    \coordinate (13) at (41.734052,87.376456);
    \coordinate (12) at (64.839772,99.999999);
    \coordinate (11) at (46.705600,71.728059);
    \coordinate (10) at (33.513026,21.203951);
    \coordinate (9) at (67.325546,15.214135);
    \coordinate (8) at (51.033243,21.638214);
    \coordinate (7) at (24.678047,45.447737);
    \coordinate (6) at (83.932315,34.980531);
    \coordinate (5) at (54.761904,42.632523);
    \coordinate (4) at (38.364778,56.349204);
    \coordinate (3) at (74.797843,49.940100);
    \coordinate (2) at (82.225216,80.248576);
    \coordinate (1) at (72.162324,67.250673);
    \draw [black] (19) to (18);
    \draw [black] (19) to (10);
    \draw [black] (18) to (17);
    \draw [black] (17) to (7);
    \draw [black] (16) to (15);
    \draw [black] (16) to (9);
    \draw [black] (15) to (14);
    \draw [black] (14) to (6);
    \draw [black] (13) to (12);
    \draw [black] (13) to (11);
    \draw [black] (12) to (2);
    \draw [black] (11) to (1);
    \draw [black] (11) to (4);
    \draw [black] (10) to (7);
    \draw [black] (10) to (8);
    \draw [black] (9) to (6);
    \draw [black] (9) to (8);
    \draw [black] (8) to (5);
    \draw [black] (7) to (4);
    \draw [black] (6) to (3);
    \draw [black] (5) to (3);
    \draw [black] (5) to (4);
    \draw [black] (3) to (1);
    \draw [black] (2) to (1);
\end{tikzpicture}
\includegraphics[scale=0.4]{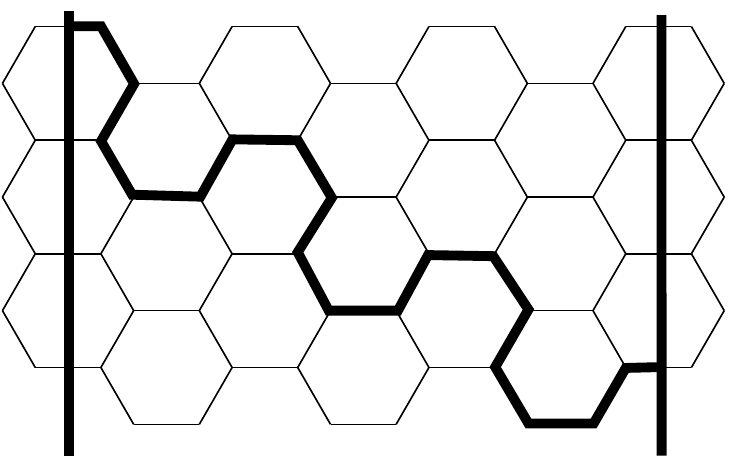}
}
\quad
\subfigure[$(6, 0)$]{
\begin{tikzpicture}[scale=0.015,thick]
    \definecolor{marked}{rgb}{0.25,0.5,0.25}
    \coordinate (18) at (49.516873,4.023652);
    \coordinate (17) at (78.850590,0.000000);
    \coordinate (16) at (90.056244,27.372366);
    \coordinate (15) at (23.197289,34.554368);
    \coordinate (14) at (76.831554,65.488893);
    \coordinate (13) at (9.943755,72.627630);
    \coordinate (12) at (21.163831,99.999999);
    \coordinate (11) at (24.653880,69.195269);
    \coordinate (10) at (53.901067,18.445342);
    \coordinate (9) at (75.389384,30.819151);
    \coordinate (8) at (53.785693,43.495816);
    \coordinate (7) at (40.546582,28.223246);
    \coordinate (6) at (73.644361,47.303143);
    \coordinate (5) at (46.271993,56.533024);
    \coordinate (4) at (26.427747,52.725698);
    \coordinate (3) at (59.482261,71.805594);
    \coordinate (2) at (50.584078,96.034033);
    \coordinate (1) at (46.127776,81.612343);
    \draw [black] (18) to (17);
    \draw [black] (18) to (10);
    \draw [black] (17) to (16);
    \draw [black] (16) to (9);
    \draw [black] (15) to (4);
    \draw [black] (15) to (7);
    \draw [black] (14) to (3);
    \draw [black] (14) to (6);
    \draw [black] (13) to (12);
    \draw [black] (13) to (11);
    \draw [black] (12) to (2);
    \draw [black] (11) to (1);
    \draw [black] (11) to (4);
    \draw [black] (10) to (9);
    \draw [black] (10) to (7);
    \draw [black] (9) to (6);
    \draw [black] (8) to (5);
    \draw [black] (8) to (6);
    \draw [black] (8) to (7);
    \draw [black] (5) to (3);
    \draw [black] (5) to (4);
    \draw [black] (3) to (1);
    \draw [black] (2) to (1);
\end{tikzpicture}
\includegraphics[scale=0.4]{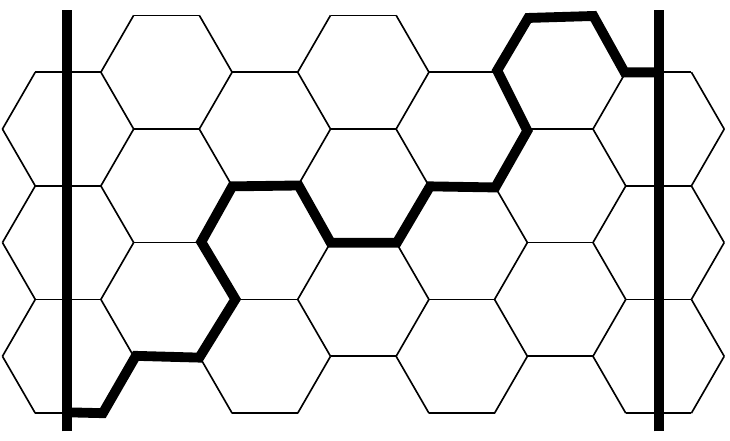}
}
\quad
\subfigure[$(6, 0)$]{
\begin{tikzpicture}[scale=0.015,thick]
    \definecolor{marked}{rgb}{0.25,0.5,0.25}
    \coordinate (18) at (16.033930,13.266583);
    \coordinate (17) at (40.356000,12.724238);
    \coordinate (16) at (0.000000,52.037268);
    \coordinate (15) at (16.854400,69.559170);
    \coordinate (14) at (83.173410,30.468641);
    \coordinate (13) at (99.999999,48.032262);
    \coordinate (12) at (59.699624,87.275761);
    \coordinate (11) at (71.589487,58.948685);
    \coordinate (10) at (12.543456,34.529271);
    \coordinate (9) at (28.410513,41.065220);
    \coordinate (8) at (44.388818,27.715198);
    \coordinate (7) at (30.315671,61.799470);
    \coordinate (6) at (54.762897,38.492559);
    \coordinate (5) at (45.264914,61.493532);
    \coordinate (4) at (69.698233,38.200528);
    \coordinate (3) at (55.625086,72.284799);
    \coordinate (2) at (83.979975,86.733416);
    \coordinate (1) at (87.442634,65.498538);
    \draw [black] (18) to (17);
    \draw [black] (18) to (10);
    \draw [black] (17) to (8);
    \draw [black] (16) to (15);
    \draw [black] (16) to (10);
    \draw [black] (15) to (7);
    \draw [black] (14) to (13);
    \draw [black] (14) to (4);
    \draw [black] (13) to (1);
    \draw [black] (12) to (2);
    \draw [black] (12) to (3);
    \draw [black] (11) to (1);
    \draw [black] (11) to (4);
    \draw [black] (11) to (3);
    \draw [black] (10) to (9);
    \draw [black] (9) to (7);
    \draw [black] (9) to (8);
    \draw [black] (8) to (6);
    \draw [black] (7) to (5);
    \draw [black] (6) to (5);
    \draw [black] (6) to (4);
    \draw [black] (5) to (3);
    \draw [black] (2) to (1);
\end{tikzpicture}
\includegraphics[scale=0.4]{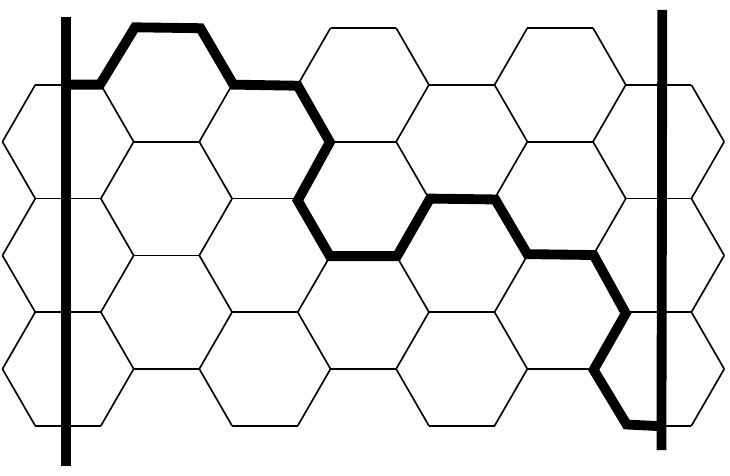}
}
\quad
\subfigure[$(4, 3)$]{
\begin{tikzpicture}[scale=0.015,thick]
    \definecolor{marked}{rgb}{0.25,0.5,0.25}
    \coordinate (19) at (22.204656,0.000000);
    \coordinate (18) at (1.408450,14.975568);
    \coordinate (17) at (10.002874,41.894222);
    \coordinate (16) at (52.069560,10.074734);
    \coordinate (15) at (88.890486,27.565392);
    \coordinate (14) at (98.591549,45.731531);
    \coordinate (13) at (44.797356,85.915492);
    \coordinate (12) at (70.465651,99.999999);
    \coordinate (11) at (49.885025,72.851394);
    \coordinate (10) at (32.466227,13.797069);
    \coordinate (9) at (24.590401,38.976717);
    \coordinate (8) at (54.512791,27.062373);
    \coordinate (7) at (37.510779,45.688416);
    \coordinate (6) at (70.652486,32.106927);
    \coordinate (5) at (66.527738,49.813164);
    \coordinate (4) at (45.515953,58.450704);
    \coordinate (3) at (84.075884,57.875826);
    \coordinate (2) at (89.149180,83.932163);
    \coordinate (1) at (76.602472,72.822649);
    \draw [black] (19) to (18);
    \draw [black] (19) to (10);
    \draw [black] (18) to (17);
    \draw [black] (17) to (9);
    \draw [black] (16) to (8);
    \draw [black] (16) to (10);
    \draw [black] (15) to (14);
    \draw [black] (15) to (6);
    \draw [black] (14) to (3);
    \draw [black] (13) to (12);
    \draw [black] (13) to (11);
    \draw [black] (12) to (2);
    \draw [black] (11) to (1);
    \draw [black] (11) to (4);
    \draw [black] (10) to (9);
    \draw [black] (9) to (7);
    \draw [black] (8) to (7);
    \draw [black] (8) to (6);
    \draw [black] (7) to (4);
    \draw [black] (6) to (5);
    \draw [black] (5) to (3);
    \draw [black] (5) to (4);
    \draw [black] (3) to (1);
    \draw [black] (2) to (1);
\end{tikzpicture}
\includegraphics[scale=0.4]{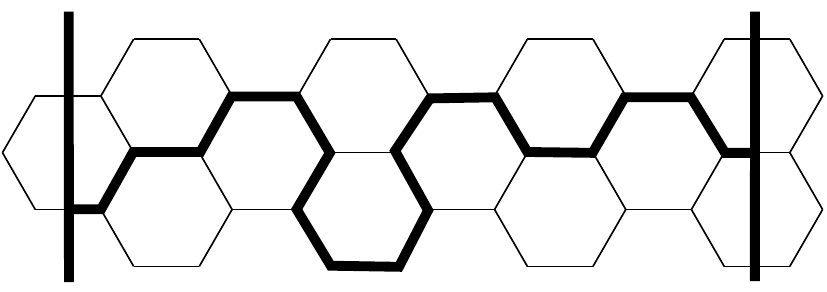}
}
\quad
\subfigure[$(4, 3)$]{
\begin{tikzpicture}[scale=0.015,thick]
    \definecolor{marked}{rgb}{0.25,0.5,0.25}
    \coordinate (18) at (54.866768,0.000000);
    \coordinate (17) at (77.647384,7.414057);
    \coordinate (16) at (28.178931,22.394037);
    \coordinate (15) at (84.094988,45.423168);
    \coordinate (14) at (16.139723,55.184316);
    \coordinate (13) at (15.905011,83.446087);
    \coordinate (12) at (40.452852,99.999999);
    \coordinate (11) at (25.514290,72.166230);
    \coordinate (10) at (60.582632,25.321001);
    \coordinate (9) at (45.423167,13.337016);
    \coordinate (8) at (79.041835,25.790418);
    \coordinate (7) at (55.474251,39.099821);
    \coordinate (6) at (35.675824,37.235953);
    \coordinate (5) at (69.018363,53.292835);
    \coordinate (4) at (32.914539,50.890516);
    \coordinate (3) at (55.985089,62.032307);
    \coordinate (2) at (64.365594,85.006213);
    \coordinate (1) at (52.878641,74.361452);
    \draw [black] (18) to (17);
    \draw [black] (18) to (9);
    \draw [black] (17) to (8);
    \draw [black] (16) to (6);
    \draw [black] (16) to (9);
    \draw [black] (15) to (5);
    \draw [black] (15) to (8);
    \draw [black] (14) to (11);
    \draw [black] (14) to (4);
    \draw [black] (13) to (12);
    \draw [black] (13) to (11);
    \draw [black] (12) to (2);
    \draw [black] (11) to (1);
    \draw [black] (10) to (7);
    \draw [black] (10) to (8);
    \draw [black] (10) to (9);
    \draw [black] (7) to (5);
    \draw [black] (7) to (6);
    \draw [black] (6) to (4);
    \draw [black] (5) to (3);
    \draw [black] (4) to (3);
    \draw [black] (3) to (1);
    \draw [black] (2) to (1);
\end{tikzpicture}
\includegraphics[scale=0.4]{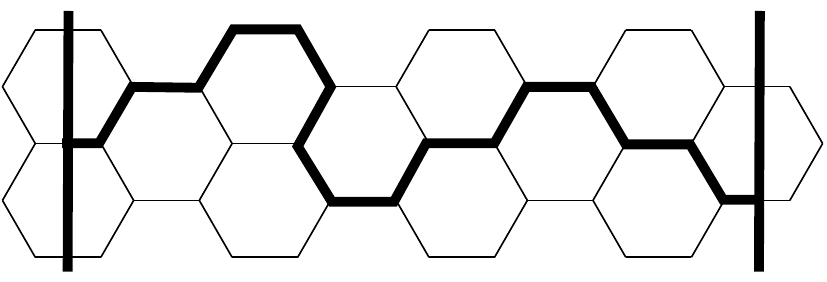}
}
\quad
\subfigure[$(5, 2)$]{
\begin{tikzpicture}[scale=0.015,thick]
    \definecolor{marked}{rgb}{0.25,0.5,0.25}
    \coordinate (19) at (45.586932,3.102443);
    \coordinate (18) at (25.186778,0.000000);
    \coordinate (17) at (15.600860,22.109660);
    \coordinate (16) at (24.667595,54.083829);
    \coordinate (15) at (75.610992,26.136507);
    \coordinate (14) at (84.399140,45.612258);
    \coordinate (13) at (42.421172,89.261747);
    \coordinate (12) at (61.972901,99.999999);
    \coordinate (11) at (49.309864,76.155502);
    \coordinate (10) at (42.484487,19.399770);
    \coordinate (9) at (25.123465,32.911232);
    \coordinate (8) at (48.942636,30.454603);
    \coordinate (7) at (37.558567,53.007471);
    \coordinate (6) at (61.985564,30.239330);
    \coordinate (5) at (60.769912,49.803723);
    \coordinate (4) at (45.979486,62.327466);
    \coordinate (3) at (73.990123,57.819425);
    \coordinate (2) at (77.801696,84.424464);
    \coordinate (1) at (68.785615,73.546916);
    \draw [black] (19) to (18);
    \draw [black] (19) to (10);
    \draw [black] (18) to (17);
    \draw [black] (17) to (9);
    \draw [black] (16) to (7);
    \draw [black] (16) to (9);
    \draw [black] (15) to (14);
    \draw [black] (15) to (6);
    \draw [black] (14) to (3);
    \draw [black] (13) to (12);
    \draw [black] (13) to (11);
    \draw [black] (12) to (2);
    \draw [black] (11) to (1);
    \draw [black] (11) to (4);
    \draw [black] (10) to (9);
    \draw [black] (10) to (8);
    \draw [black] (8) to (7);
    \draw [black] (8) to (6);
    \draw [black] (7) to (4);
    \draw [black] (6) to (5);
    \draw [black] (5) to (3);
    \draw [black] (5) to (4);
    \draw [black] (3) to (1);
    \draw [black] (2) to (1);
\end{tikzpicture}
\includegraphics[scale=0.4]{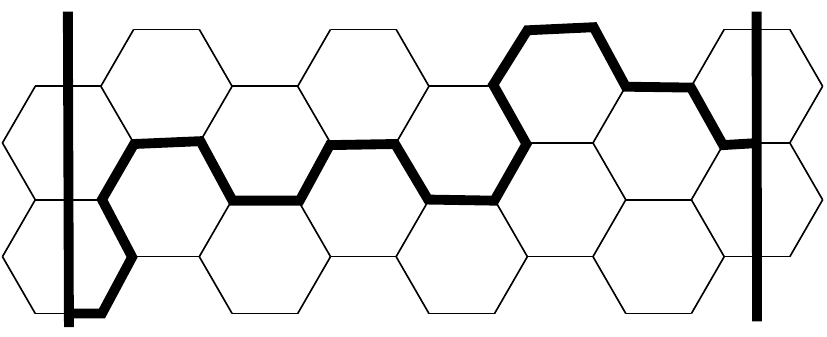}
}
\quad
\subfigure[$(5, 2)$]{
\begin{tikzpicture}[scale=0.015,thick]
    \definecolor{marked}{rgb}{0.25,0.5,0.25}
    \coordinate (19) at (54.167214,0.000000);
    \coordinate (18) at (36.562045,10.272976);
    \coordinate (17) at (86.568640,21.561387);
    \coordinate (16) at (83.904785,44.903071);
    \coordinate (15) at (22.768033,41.461161);
    \coordinate (14) at (13.431359,70.090992);
    \coordinate (13) at (23.677963,96.136093);
    \coordinate (12) at (53.389158,99.999999);
    \coordinate (11) at (26.038505,82.144268);
    \coordinate (10) at (68.501910,15.046813);
    \coordinate (9) at (61.143345,28.919950);
    \coordinate (8) at (43.327177,25.003296);
    \coordinate (7) at (69.385466,47.092178);
    \coordinate (6) at (38.487405,38.441249);
    \coordinate (5) at (56.382697,53.132004);
    \coordinate (4) at (25.365948,58.182776);
    \coordinate (3) at (49.419753,63.615982);
    \coordinate (2) at (64.862189,83.858631);
    \coordinate (1) at (51.595673,75.748382);
    \draw [black] (19) to (18);
    \draw [black] (19) to (10);
    \draw [black] (18) to (8);
    \draw [black] (17) to (16);
    \draw [black] (17) to (10);
    \draw [black] (16) to (7);
    \draw [black] (15) to (4);
    \draw [black] (15) to (6);
    \draw [black] (14) to (11);
    \draw [black] (14) to (4);
    \draw [black] (13) to (12);
    \draw [black] (13) to (11);
    \draw [black] (12) to (2);
    \draw [black] (11) to (1);
    \draw [black] (10) to (9);
    \draw [black] (9) to (7);
    \draw [black] (9) to (8);
    \draw [black] (8) to (6);
    \draw [black] (7) to (5);
    \draw [black] (6) to (5);
    \draw [black] (5) to (3);
    \draw [black] (4) to (3);
    \draw [black] (3) to (1);
    \draw [black] (2) to (1);
\end{tikzpicture}
\includegraphics[scale=0.4]{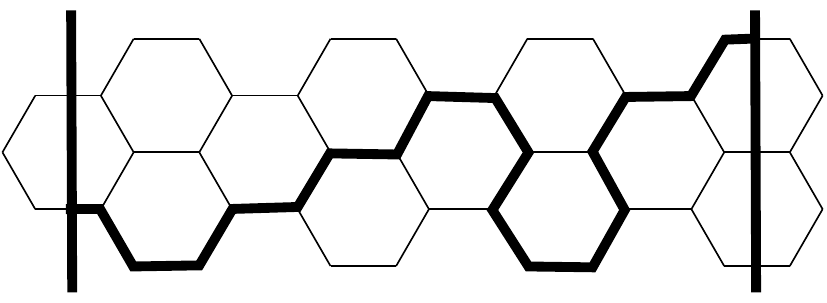}
}
\quad
\subfigure[$(6, 1)$]{
\begin{tikzpicture}[scale=0.015,thick]
    \coordinate (20) at (55.727014,6.262909);
    \coordinate (19) at (72.991993,0.000000);
    \coordinate (18) at (85.117510,18.620864);
    \coordinate (17) at (79.745611,50.309916);
    \coordinate (16) at (35.556559,30.746383);
    \coordinate (15) at (16.509554,53.267044);
    \coordinate (14) at (14.882489,80.810948);
    \coordinate (13) at (42.865443,99.999999);
    \coordinate (12) at (70.086518,98.450414);
    \coordinate (11) at (36.847880,87.009296);
    \coordinate (10) at (60.388688,22.301134);
    \coordinate (9) at (77.782800,31.572829);
    \coordinate (8) at (51.930526,32.773758);
    \coordinate (7) at (64.314307,51.071796);
    \coordinate (6) at (32.366992,46.151858);
    \coordinate (5) at (52.924844,57.541321);
    \coordinate (4) at (21.933109,67.432850);
    \coordinate (3) at (48.870092,67.781507);
    \coordinate (2) at (75.729597,82.838325);
    \coordinate (1) at (57.547778,77.789256);
    \draw [black] (20) to (19);
    \draw [black] (20) to (10);
    \draw [black] (19) to (18);
    \draw [black] (18) to (9);
    \draw [black] (17) to (7);
    \draw [black] (17) to (9);
    \draw [black] (16) to (6);
    \draw [black] (16) to (8);
    \draw [black] (15) to (4);
    \draw [black] (15) to (6);
    \draw [black] (14) to (11);
    \draw [black] (14) to (4);
    \draw [black] (13) to (12);
    \draw [black] (13) to (11);
    \draw [black] (12) to (2);
    \draw [black] (11) to (1);
    \draw [black] (10) to (9);
    \draw [black] (10) to (8);
    \draw [black] (8) to (7);
    \draw [black] (7) to (5);
    \draw [black] (6) to (5);
    \draw [black] (5) to (3);
    \draw [black] (4) to (3);
    \draw [black] (3) to (1);
    \draw [black] (2) to (1);
\end{tikzpicture}
\includegraphics[scale=0.4]{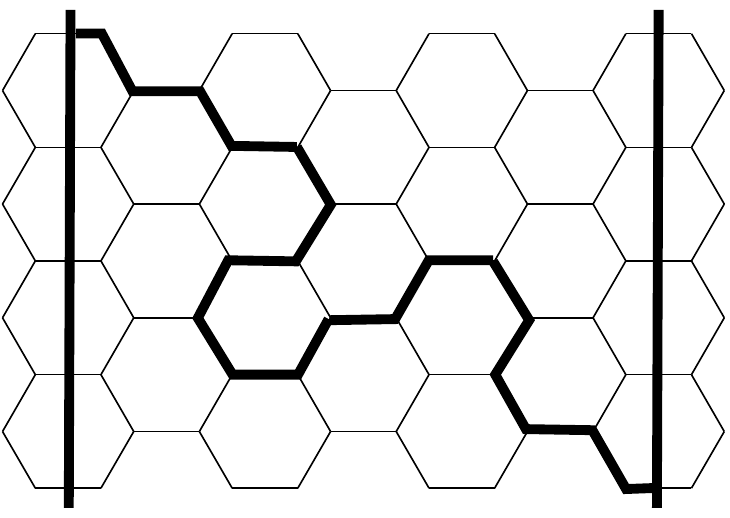}
}
\quad
\subfigure[$(6, 1)$]{
\begin{tikzpicture}[scale=0.015,thick]
    \definecolor{marked}{rgb}{0.25,0.5,0.25}
    \coordinate (19) at (40.443858,0.000000);
    \coordinate (18) at (24.261285,15.025776);
    \coordinate (17) at (79.171383,14.422229);
    \coordinate (16) at (81.698731,35.143970);
    \coordinate (15) at (64.900038,60.430026);
    \coordinate (14) at (18.301269,51.766629);
    \coordinate (13) at (21.645921,84.169497);
    \coordinate (12) at (40.179807,99.999999);
    \coordinate (11) at (25.543820,70.325663);
    \coordinate (10) at (58.185591,10.612346);
    \coordinate (9) at (53.269207,23.311958);
    \coordinate (8) at (34.659877,25.675845);
    \coordinate (7) at (67.100467,36.225323);
    \coordinate (6) at (37.690180,38.614359);
    \coordinate (5) at (59.556144,46.271847);
    \coordinate (4) at (31.290080,48.145354);
    \coordinate (3) at (49.446750,63.397460);
    \coordinate (2) at (55.909720,88.582924);
    \coordinate (1) at (45.599146,75.392932);
    \draw [black] (19) to (18);
    \draw [black] (19) to (10);
    \draw [black] (18) to (8);
    \draw [black] (17) to (16);
    \draw [black] (17) to (10);
    \draw [black] (16) to (7);
    \draw [black] (15) to (3);
    \draw [black] (15) to (5);
    \draw [black] (14) to (11);
    \draw [black] (14) to (4);
    \draw [black] (13) to (12);
    \draw [black] (13) to (11);
    \draw [black] (12) to (2);
    \draw [black] (11) to (1);
    \draw [black] (10) to (9);
    \draw [black] (9) to (7);
    \draw [black] (9) to (8);
    \draw [black] (8) to (6);
    \draw [black] (7) to (5);
    \draw [black] (6) to (5);
    \draw [black] (6) to (4);
    \draw [black] (4) to (3);
    \draw [black] (3) to (1);
    \draw [black] (2) to (1);
\end{tikzpicture}
\includegraphics[scale=0.4]{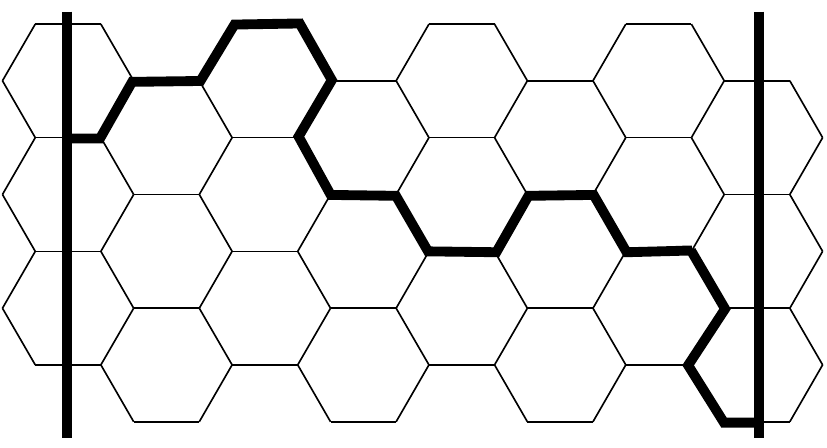}
}
\quad
\subfigure[$(7, 0)$]{
\begin{tikzpicture}[scale=0.015,thick]
    \definecolor{marked}{rgb}{0.25,0.5,0.25}
    \coordinate (20) at (20.708748,15.946842);
    \coordinate (19) at (32.048726,0.000000);
    \coordinate (18) at (50.033223,5.359911);
    \coordinate (17) at (29.590255,40.166112);
    \coordinate (16) at (71.262458,28.815060);
    \coordinate (15) at (79.291252,53.654485);
    \coordinate (14) at (38.914729,68.903655);
    \coordinate (13) at (45.935769,93.743077);
    \coordinate (12) at (64.440753,99.999999);
    \coordinate (11) at (48.316722,80.719822);
    \coordinate (10) at (30.276854,25.071982);
    \coordinate (9) at (46.744186,17.995568);
    \coordinate (8) at (43.875969,41.040974);
    \coordinate (7) at (55.825028,27.198227);
    \coordinate (6) at (52.624585,48.527132);
    \coordinate (5) at (71.561462,42.436323);
    \coordinate (4) at (49.922481,59.745293);
    \coordinate (3) at (67.729791,63.975636);
    \coordinate (2) at (75.869324,85.105206);
    \coordinate (1) at (65.791805,76.777409);
    \draw [black] (20) to (19);
    \draw [black] (20) to (10);
    \draw [black] (19) to (18);
    \draw [black] (18) to (9);
    \draw [black] (17) to (8);
    \draw [black] (17) to (10);
    \draw [black] (16) to (5);
    \draw [black] (16) to (7);
    \draw [black] (15) to (3);
    \draw [black] (15) to (5);
    \draw [black] (14) to (11);
    \draw [black] (14) to (4);
    \draw [black] (13) to (12);
    \draw [black] (13) to (11);
    \draw [black] (12) to (2);
    \draw [black] (11) to (1);
    \draw [black] (10) to (9);
    \draw [black] (9) to (7);
    \draw [black] (8) to (7);
    \draw [black] (8) to (6);
    \draw [black] (6) to (5);
    \draw [black] (6) to (4);
    \draw [black] (4) to (3);
    \draw [black] (3) to (1);
    \draw [black] (2) to (1);
\end{tikzpicture}
\includegraphics[scale=0.4]{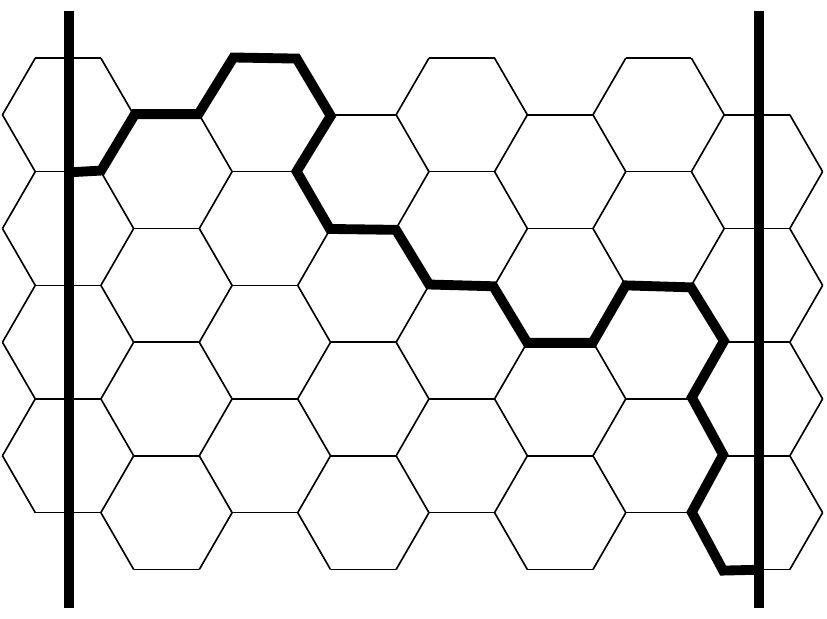}
}
\quad
\subfigure[$(4, 4)$]{
\begin{tikzpicture}[scale=0.015,thick]
    \definecolor{marked}{rgb}{0.25,0.5,0.25}
    \coordinate (20) at (28.062968,15.602477);
    \coordinate (19) at (35.560754,0.568973);
    \coordinate (18) at (63.427740,0.000000);
    \coordinate (17) at (78.524465,23.049689);
    \coordinate (16) at (74.617524,52.585661);
    \coordinate (15) at (25.395120,47.363762);
    \coordinate (14) at (21.475534,76.849158);
    \coordinate (13) at (36.483754,99.999999);
    \coordinate (12) at (64.249590,99.431026);
    \coordinate (11) at (34.675687,86.863067);
    \coordinate (10) at (44.145910,21.266913);
    \coordinate (9) at (65.248452,13.086358);
    \coordinate (8) at (48.128715,33.101529);
    \coordinate (7) at (70.242761,36.553293);
    \coordinate (6) at (40.984955,43.355671);
    \coordinate (5) at (59.052978,56.606397);
    \coordinate (4) at (29.820459,63.332911);
    \coordinate (3) at (51.921862,66.873181);
    \coordinate (2) at (71.899103,84.422808);
    \coordinate (1) at (55.828803,78.720444);
    \draw [black] (20) to (19);
    \draw [black] (20) to (10);
    \draw [black] (19) to (18);
    \draw [black] (18) to (9);
    \draw [black] (17) to (7);
    \draw [black] (17) to (9);
    \draw [black] (16) to (5);
    \draw [black] (16) to (7);
    \draw [black] (15) to (4);
    \draw [black] (15) to (6);
    \draw [black] (14) to (11);
    \draw [black] (14) to (4);
    \draw [black] (13) to (12);
    \draw [black] (13) to (11);
    \draw [black] (12) to (2);
    \draw [black] (11) to (1);
    \draw [black] (10) to (9);
    \draw [black] (10) to (8);
    \draw [black] (8) to (7);
    \draw [black] (8) to (6);
    \draw [black] (6) to (5);
    \draw [black] (5) to (3);
    \draw [black] (4) to (3);
    \draw [black] (3) to (1);
    \draw [black] (2) to (1);
\end{tikzpicture}
\includegraphics[scale=0.4]{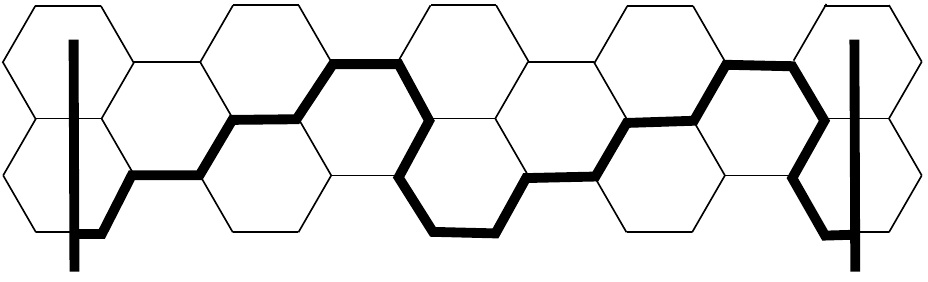}
}
\quad
\subfigure[$(5, 3)$]{
\begin{tikzpicture}[scale=0.015,thick]
    \definecolor{marked}{rgb}{0.25,0.5,0.25}
    \coordinate (20) at (65.262656,0.000000);
    \coordinate (19) at (89.609416,5.960805);
    \coordinate (18) at (91.174467,27.191071);
    \coordinate (17) at (39.133098,14.833967);
    \coordinate (16) at (72.611594,54.559062);
    \coordinate (15) at (19.454273,38.704408);
    \coordinate (14) at (8.825532,65.949917);
    \coordinate (13) at (21.005716,94.433857);
    \coordinate (12) at (53.871802,99.999999);
    \coordinate (11) at (21.863092,80.307566);
    \coordinate (10) at (57.981764,11.554164);
    \coordinate (9) at (74.652967,26.959716);
    \coordinate (8) at (45.815188,30.035382);
    \coordinate (7) at (67.780349,38.486662);
    \coordinate (6) at (38.017148,40.242242);
    \coordinate (5) at (55.287153,56.573216);
    \coordinate (4) at (20.828797,55.538921);
    \coordinate (3) at (45.298040,64.847576);
    \coordinate (2) at (65.072129,85.655958);
    \coordinate (1) at (47.638813,77.150243);
    \draw [black] (20) to (19);
    \draw [black] (20) to (10);
    \draw [black] (19) to (18);
    \draw [black] (18) to (9);
    \draw [black] (17) to (8);
    \draw [black] (17) to (10);
    \draw [black] (16) to (5);
    \draw [black] (16) to (7);
    \draw [black] (15) to (4);
    \draw [black] (15) to (6);
    \draw [black] (14) to (11);
    \draw [black] (14) to (4);
    \draw [black] (13) to (12);
    \draw [black] (13) to (11);
    \draw [black] (12) to (2);
    \draw [black] (11) to (1);
    \draw [black] (10) to (9);
    \draw [black] (9) to (7);
    \draw [black] (8) to (7);
    \draw [black] (8) to (6);
    \draw [black] (6) to (5);
    \draw [black] (5) to (3);
    \draw [black] (4) to (3);
    \draw [black] (3) to (1);
    \draw [black] (2) to (1);
\end{tikzpicture}
\includegraphics[scale=0.4]{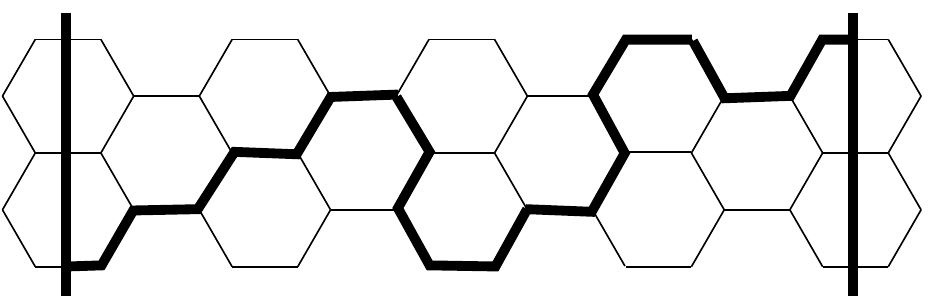}
}
\quad
\subfigure[$(6, 2)$]{
\begin{tikzpicture}[scale=0.015,thick]
    \definecolor{marked}{rgb}{0.25,0.5,0.25}
    \coordinate (20) at (35.728335,10.644453);
    \coordinate (19) at (51.856994,0.000000);
    \coordinate (18) at (70.704616,14.046048);
    \coordinate (17) at (75.494618,39.743145);
    \coordinate (16) at (30.614370,35.427515);
    \coordinate (15) at (69.397199,64.549346);
    \coordinate (14) at (24.505381,60.256856);
    \coordinate (13) at (29.202823,85.930810);
    \coordinate (12) at (48.062015,99.999999);
    \coordinate (11) at (33.182922,75.448340);
    \coordinate (10) at (46.615758,20.814532);
    \coordinate (9) at (66.782367,24.540090);
    \coordinate (8) at (43.862085,30.984612);
    \coordinate (7) at (62.281614,42.161287);
    \coordinate (6) at (40.194377,47.830615);
    \coordinate (5) at (59.782482,52.134676);
    \coordinate (4) at (37.706815,57.804003);
    \coordinate (3) at (56.126345,69.003818);
    \coordinate (2) at (64.213814,89.355547);
    \coordinate (1) at (53.326391,79.185468);
    \draw [black] (20) to (19);
    \draw [black] (20) to (10);
    \draw [black] (19) to (18);
    \draw [black] (18) to (9);
    \draw [black] (17) to (7);
    \draw [black] (17) to (9);
    \draw [black] (16) to (6);
    \draw [black] (16) to (8);
    \draw [black] (15) to (3);
    \draw [black] (15) to (5);
    \draw [black] (14) to (11);
    \draw [black] (14) to (4);
    \draw [black] (13) to (12);
    \draw [black] (13) to (11);
    \draw [black] (12) to (2);
    \draw [black] (11) to (1);
    \draw [black] (10) to (9);
    \draw [black] (10) to (8);
    \draw [black] (8) to (7);
    \draw [black] (7) to (5);
    \draw [black] (6) to (5);
    \draw [black] (6) to (4);
    \draw [black] (4) to (3);
    \draw [black] (3) to (1);
    \draw [black] (2) to (1);
\end{tikzpicture}
\includegraphics[scale=0.4]{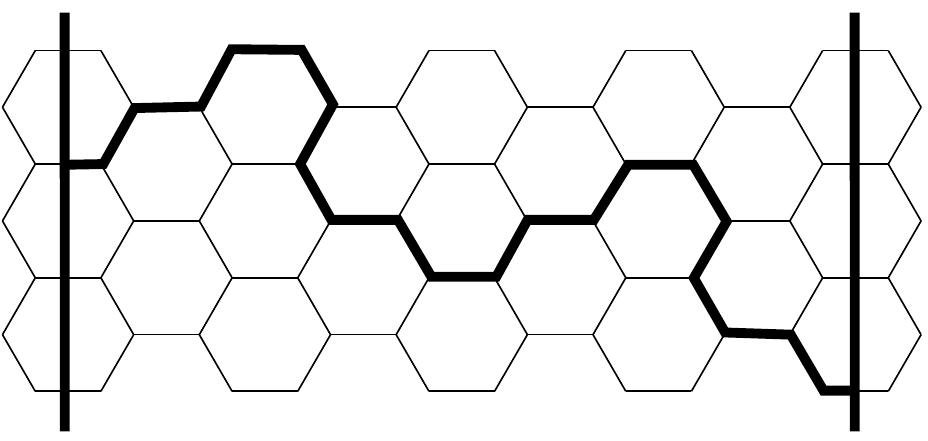}
}
\caption{All possible clusters containing 6 pentagons, with mappings
  of their boundaries into the hexagonal grid, and the tube parameters
  to which they correspond.}
\label{fig:p6}
\end{figure}
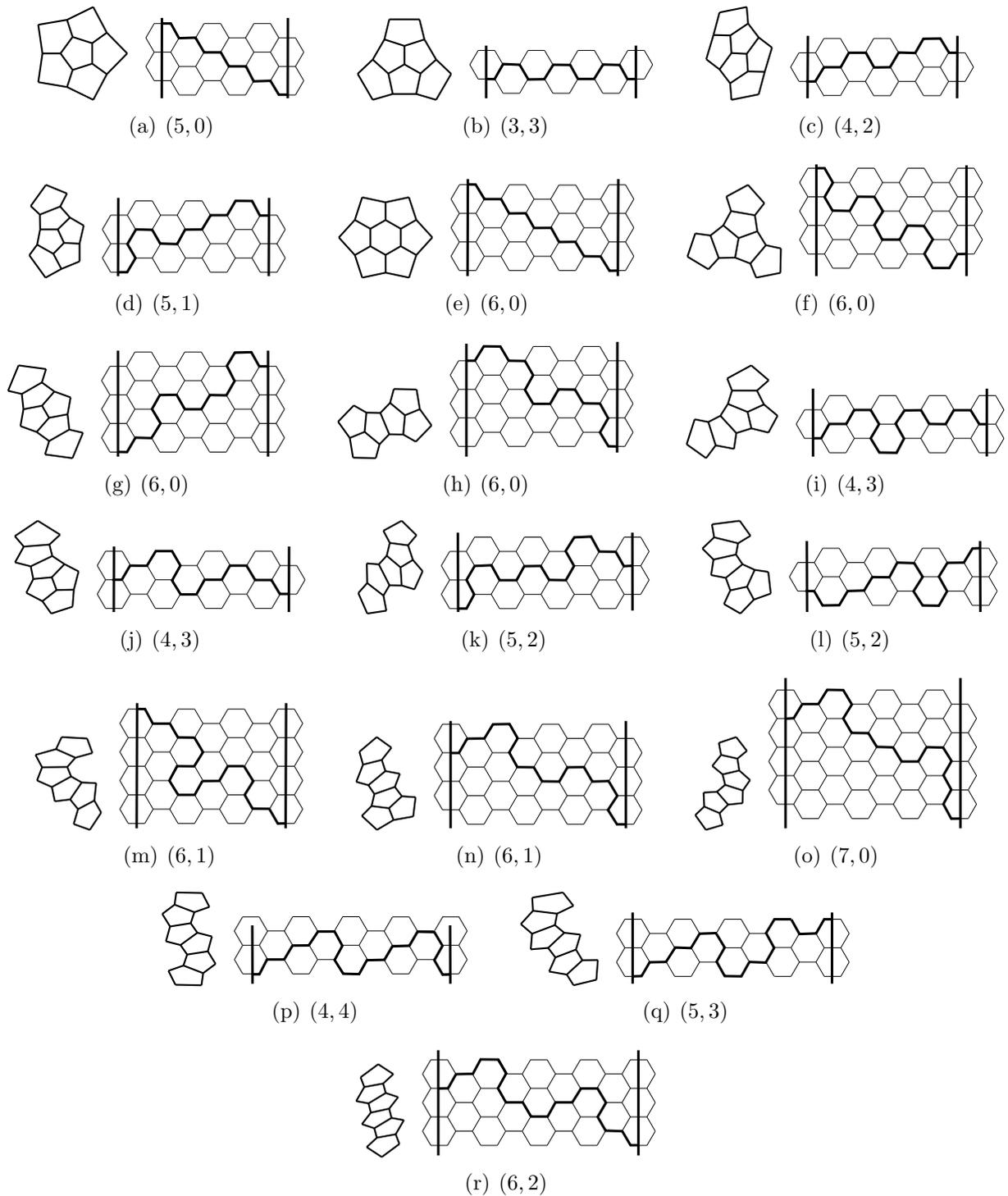

In order to identify all possible combinations including a pentagon
cluster of size 6 we only have to look at all tube caps with these
parameters.  We did this by computer.  We searched the output of the
program {\em tube} (see \cite{BNP99}) by two independent programs and
classified the caps that were generated 
based on the clusters that were found.  The results of these
searches are summarised in Table~\ref{tab:summary}.

\section{Partitions $\boldsymbol{(p_1,p_2,\dots ,p_k)}$ with $\boldsymbol{p_1<6}$}

\begin{definition}

A {\em hexagon cycle } 
in a fullerene is a cyclic sequence $h_0,\dots ,h_{n-1}$ 
of different hexagons so that for $0\le i< n$ hexagon $h_i$ shares an edge with
hexagon $h_{i+1}$ and an edge with hexagon $h_{i-1}$ where all subscripts are taken modulo $n$.

\end{definition}

We can draw a Jordan curve through the faces of a hexagon cycle in the given order.
This curve splits the set of faces of the fullerene into two parts that we will call the {\em outside} and 
the {\em inside} of the cycle.

If for two pentagons $p_1,p_2$ in a fullerene $F$ there are $k$ disjoint hexagon cycles with $p_1$ inside each of
the cycles and $p_2$ outside, then the Jordan curve theorem implies $d(p_1,p_2)\ge k$.

Our main aim in this part of the work is to prove that each partition $(p_1,p_2,\dots ,p_k)$ with $p_1<6$ is in class (d).
This will be a direct consequence of the following theorem:

\begin{theorem}\label{thm:atmost5}
Let $P = (p_1,p_2,…,p_k)$ be a partition of 12 with $p_i < 6$ for $1\le i \le k$  and let $s>0$.
Then there exists a fullerene $F$ so that $\mathit{PIP}(F) =  P$, $s(F)\ge s$ and each pentagon cluster in 
$F$ is one of the clusters depicted in Figure~\ref{fig:classes}.

\end{theorem}

\begin{figure}[htp]
\begin{center}
\includegraphics[width=130mm]{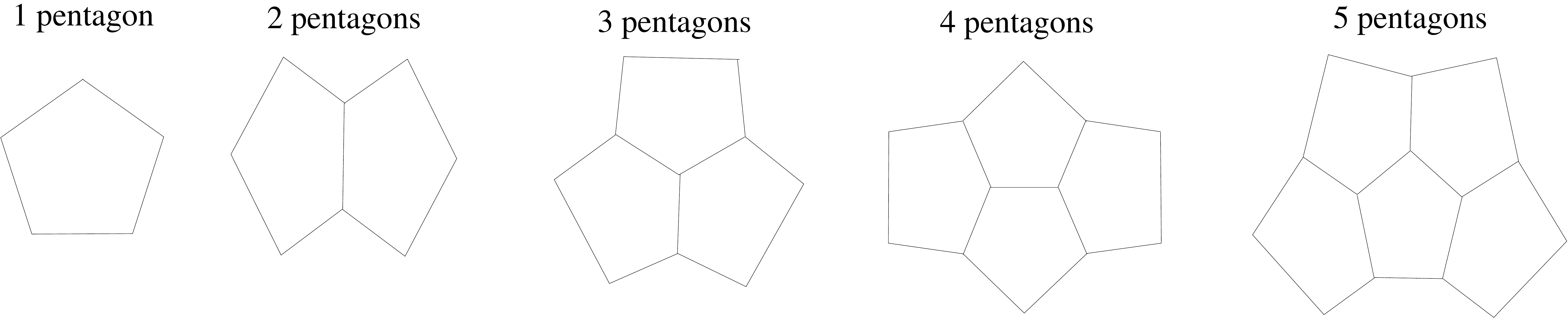}
\caption{\label{fig:classes} The clusters chosen to represent the classes with $1$ to $5$ pentagons. The
clusters are chosen for their symmetry and simplicity, but other choices would work equally
well.
}
\end{center}
\end{figure}

We will define an operation that when given a fullerene with representative clusters $C_1, C_2, \dots ,C_k$
allows construction of fullerenes containing the same clusters, but with more separating hexagon cycles. We use the Goldberg
operation of type $(5,0)$. It is the smallest Goldberg operation \cite{Coxeter_virus}
preserving all 
equivalence classes of cones (see \cite{cones_article,Doug_Balaban_cone}). 
Although the construction given here is based on this observation, details of the 
effect of Goldberg operations on equivalence classes of cones 
are not needed to follow the arguments.

The easiest way to describe the Goldberg operation $(5,0)$ on a fullerene is
given in Figure~\ref{fig:5_0}. By gluing those patterns into each pentagon, resp.\ hexagon, one gets a new fullerene.
It is easy to see that in the new fullerene formerly adjacent pentagons produce pentagons at distance 5.
In the sequel, it will be also helpful to remember the previous underlying fullerene structure.

\begin{figure}[htbp]
\centering
\includegraphics[width=80mm]{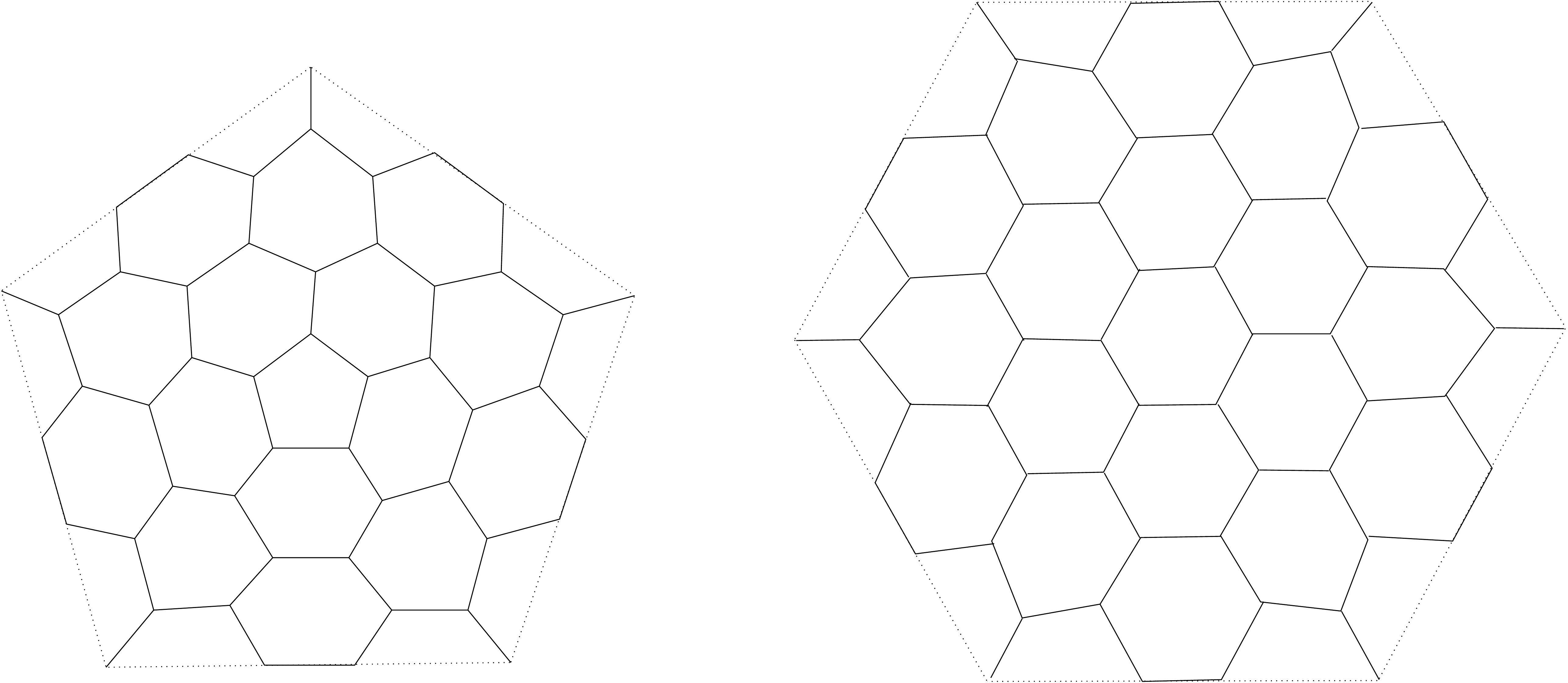}
\caption{\label{fig:5_0} The patterns to be inserted into pentagons (left) and hexagons (right) for the Goldberg operation $(5,0)$.}
\end{figure}

\begin{figure}[htbp]
\centering
\includegraphics[width=120mm]{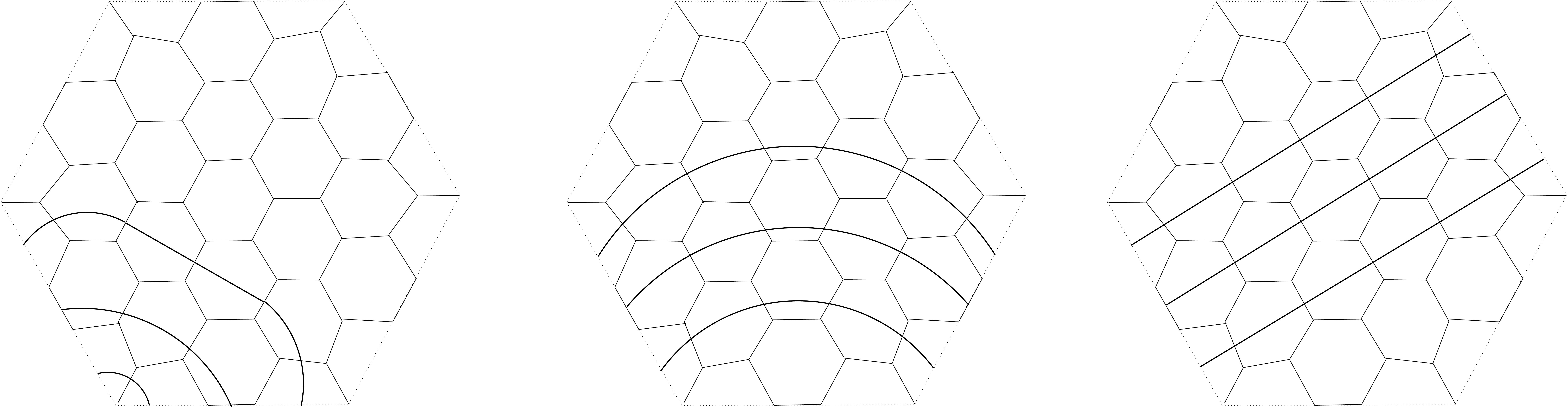}
\caption{\label{fig:traversal} The three ways a hexagon in a hexagon
  cycle can be traversed, showing how the three hexagon cycles emerging
  from the original hexagon can be defined.}
\end{figure}

In Figure~\ref{fig:traversal} it is shown how a hexagon cycle before the application of the Goldberg operation gives rise
to three hexagon cycles afterwards.  New faces that are contained in a face that was inside the old hexagon cycle
lie inside all three new hexagon cycles. The definition of pentagon clusters in a fullerene implies that each pair of clusters is separated
by at least one hexagon cycle, and on applying the Goldberg operation $(5,0)$ $k$ times, the faces lying inside different original
clusters awill become separated by at least  $3^k$ hexagon cycles and therefore have distance at least $3^k$. This remains true
if we apply modifications to the results of the operations,
provided they remain inside the region of the former clusters.

 After applying the Goldberg operation, all pentagons will be isolated, so the original clusters will be destroyed.
Figures~\ref{fig:replace2} to \ref{fig:replace5} show how parts of the inflated fullerene can be replaced
(inside the perimeter of the original cluster) so that the original clusters are reinstated.

\begin{figure}[!bp]
\centering
\includegraphics[width=80mm]{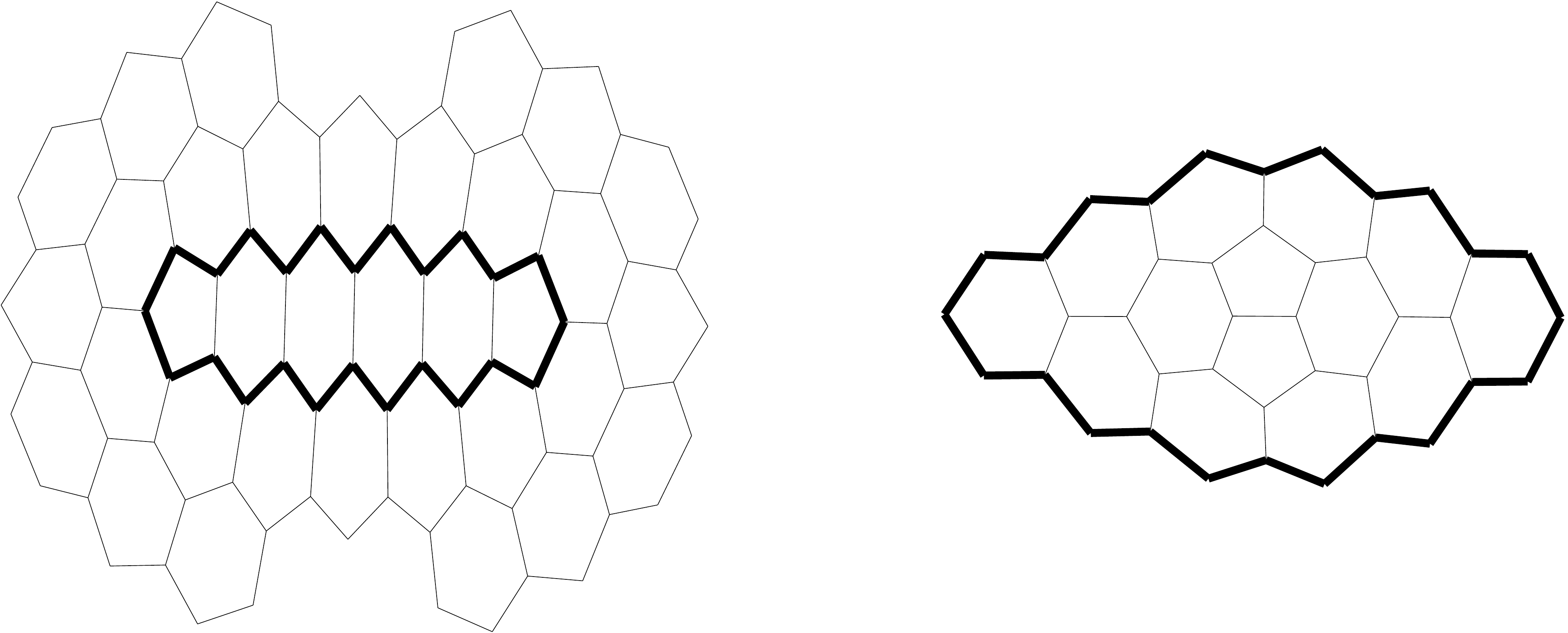}
\caption{\label{fig:replace2} After applying the Goldberg operation $(5,0)$, the inflated cluster with 2 pentagons has isolated pentagons.
It is replaced by the original with some 
additional surrounding hexagons. The faces shown for the inflated cluster all lie inside the original cluster.}
\end{figure}

\begin{figure}[htp]
\begin{center}
\includegraphics[width=100mm]{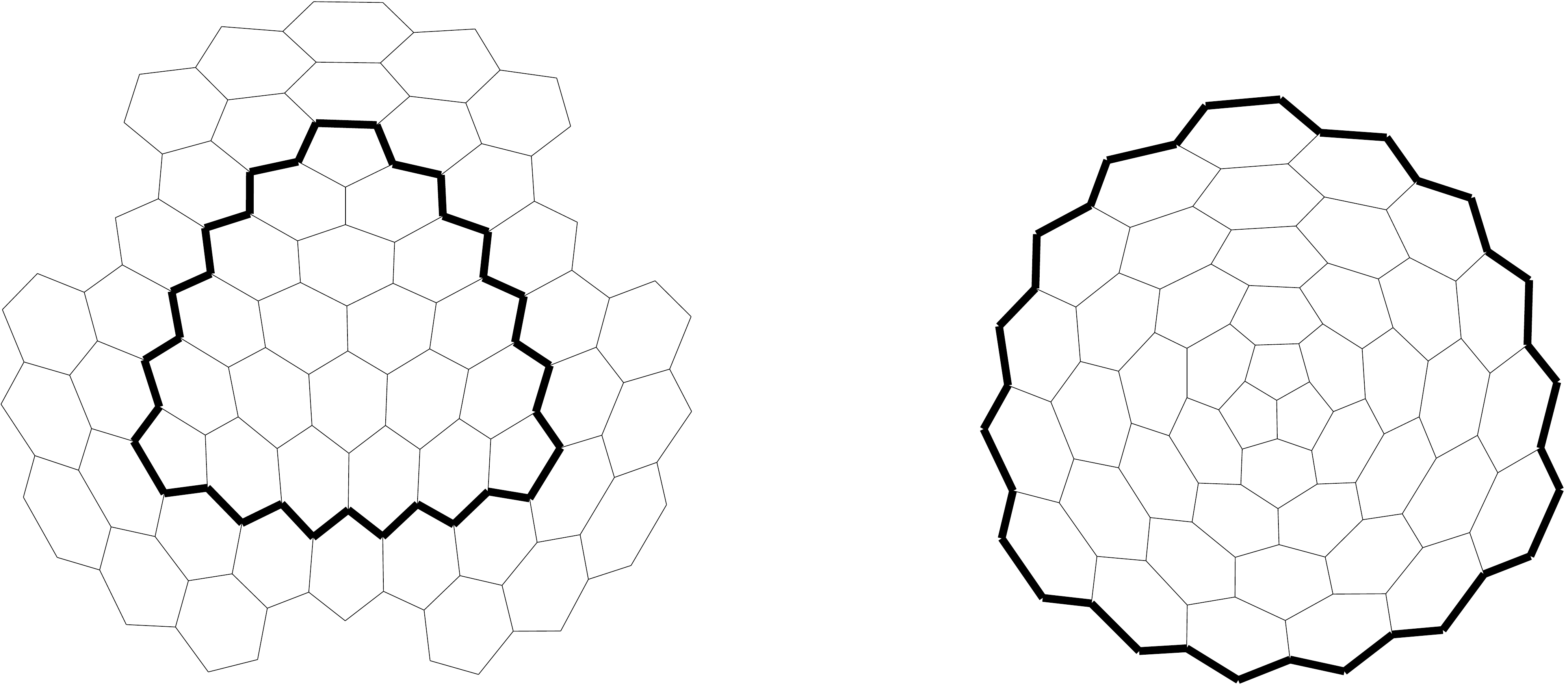}
\caption{\label{fig:replace3b} The inflated cluster with 3 pentagons, and its replacement.}
\end{center}
\end{figure}

\begin{figure}[htp]
\begin{center}
\includegraphics[width=120mm]{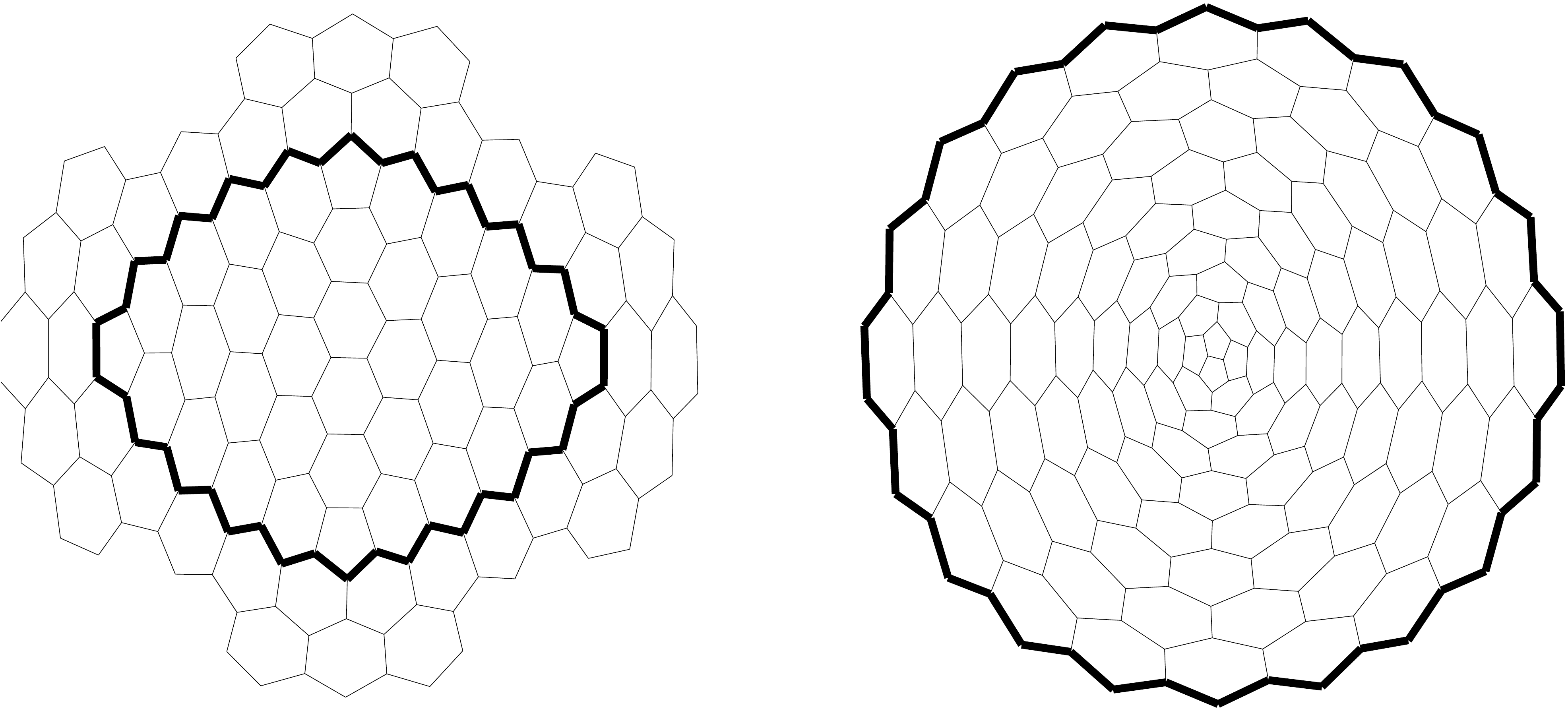}
\caption{\label{fig:replace4b} The inflated cluster with 4 pentagons, and its replacement.}
\end{center}
\end{figure}

\begin{figure}[htp]
\begin{center}
\includegraphics[width=140mm]{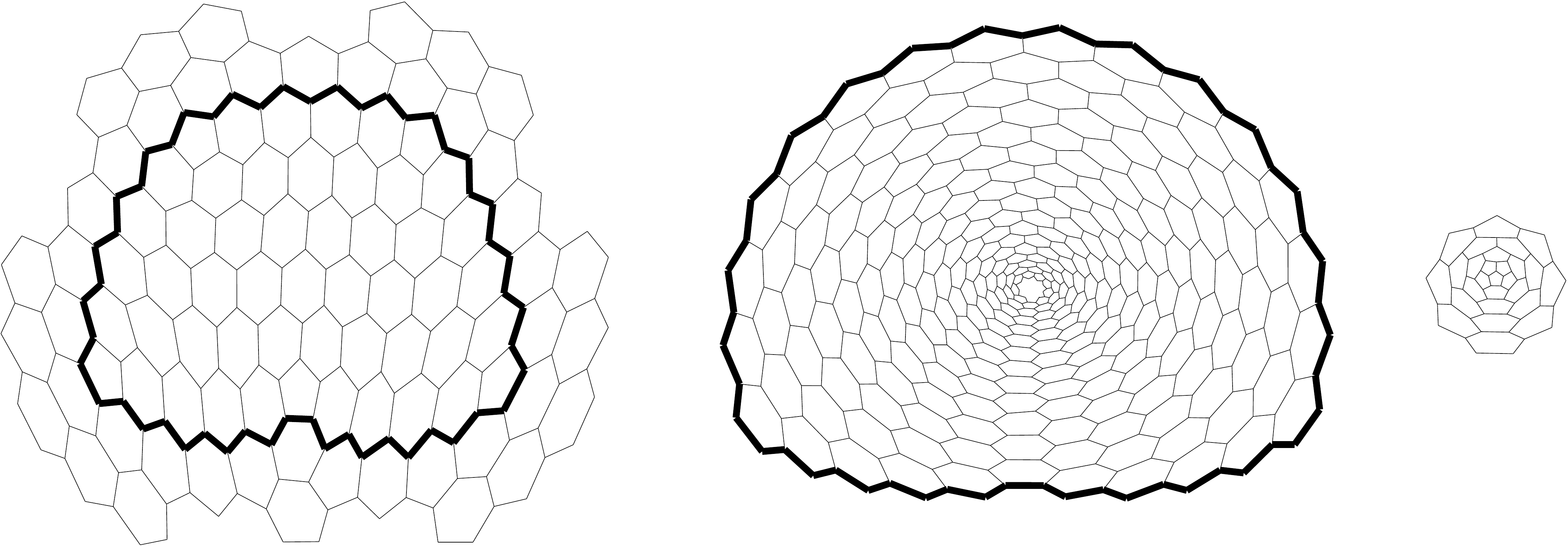}
\caption{\label{fig:replace5} The inflated cluster with 5 pentagons, and its replacement. 
The innermost part of the replacement is drawn
separately for better visibility.}
\end{center}
\end{figure}

In order to complete the proof of Theorem~\ref{thm:atmost5}, we
must finally show that for each combination of our representative
clusters that include altogether 12 pentagons, there exists a fullerene with
these clusters that we can use as a starting point of our inflation
procedure. We did this by computer. We searched the output of the
program {\em fullgen} (see \cite{BrDr95_2}) by two independent
programs for fullerenes with the given clusters. For all combinations
of representative clusters with together $12$ pentagons a fullerene 
with this combination of clusters was found. For the combination
of 7 isolated vertices and the 5-cluster, the largest number of vertices
was needed. The smallest fullerene with this combination has $100$ vertices.

A direct consequence of Theorem~\ref{thm:atmost5} is the following corollary:

\begin{corollary}\label{cor:atmost5}

Each partition $(p_1,p_2,\dots ,p_k)$ with $p_1<6$ is in class (d).

\end{corollary}

\section{Conclusion}

For each possible partition of the 
number $12$ we have decided whether this partition can describe the
sizes of pentagon clusters in a fullerene and whether the different clusters
can be at an arbitrarily large distance from each other.

Of the $30$ possible partitions with largest cluster of size 6 or
more, $15$ are impossible in a fullerene, $9$ occur in only a finite
number of fullerene isomers, $5$ occur in an infinite number of
fullerenes but with bounded separation number, and only $1$ (partition
into two sets of 6) occurs in an infinite number of fullerenes and
with unbounded separation as defined by the separation number.  All
$47$ partitions with largest cluster of size $5$ or less can occur in
an infinite number of fullerenes and have unbounded separation number.

Here we have focused on the sizes of clusters rather than their exact structure,
as the number of combinations of non-isomorphic clusters is simply too large. The fact
that a certain combination of cluster sizes can occur does not imply that all clusters
with these sizes can occur together in a fullerene. Investigating the most important
and interesting combinations of cluster sizes in detail might be the topic of further research.

Perhaps the main chemical significance of the results is that they
rule out so many apparently possible types of pentagon distribution
for fullerenes: fullerenes with one connected set of pentagons of
large size, and various isolated pentagon pairs and singletons, are
simply impossible in many cases: no fullerene can contain a cluster of
$p > 6$ pentagons and three isolated pentagons, for example 
(see Table~\ref{tab:summary}).  The fullerenes that are in some sense furthest
from the IPR class, with one or more large ($p > 6$) connected
components of pentagons, have been characterised: there are $41$ with
a $12$-cluster, $2$ with an $11$ cluster, and so on. Table~\ref{tab:spiralnumbers}
lists these isomers.  Some have the minimum number of pentagon
adjacencies for their vertex count, but many others are, of course,
energetically unlikely as bare neutral cages. This feature makes
them attractive as a test set for investigation of the factors that
determine fullerene stability, and for testing how far the stability
rules can be bent by changing electron count, for example.

\providecommand{\bysame}{\leavevmode\hbox to3em{\hrulefill}\thinspace}
\providecommand{\MR}{\relax\ifhmode\unskip\space\fi MR }
\providecommand{\MRhref}[2]{%
  \href{http://www.ams.org/mathscinet-getitem?mr=#1}{#2}
}
\providecommand{\href}[2]{#2}

\end{document}